\documentclass[letter,final,oneside,notitlepage,11pt]{article}

\usepackage{amsthm}
\usepackage{graphicx}
\usepackage[margin=2cm,font=normalsize,labelfont=bf]{caption}
\usepackage{verbatim}
\usepackage{enumerate}
\usepackage{fancyhdr}
\usepackage[]{geometry}
\usepackage[normalem]{ulem}
\usepackage{xstring}
\usepackage{needspace}

\makeatletter

\newcounter{denseversion}
\newcounter{authorcounter}
\newcounter{adresscounter}

\def\title#1{\gdef\@title{#1}}
\def\@title{}
\def\subtitle#1{\gdef\@subtitle{#1}}
\def\@subtitle{}

\def\authortagsused{0}
\def\adresstag#1{\if!#1!\else$^{\;#1\;}$\fi}

\renewcommand{\author}[2][]{
  \stepcounter{authorcounter}
  \if!#1!\else\gdef\authortagsused{1}\fi
  \ifnum\value{authorcounter}=1
    \def\@authorstringa{#2\adresstag{#1}}
    \def\@authorstringb{#2}
    \def\@authorstringc{#2\adresstag{#1}}
  \else
    \g@addto@macro\@authorstringa{\ and #2\adresstag{#1}}
    \g@addto@macro\@authorstringb{\ and #2}
    \g@addto@macro\@authorstringc{\\#2\adresstag{#1}}
  \fi}
\def\@author{\ifnum\value{denseversion}=0\@authorstringa\else\@authorstringb\fi}

\def\@adressstringa{}
\def\@adressstringb{}
\newcommand{\adress}[2][]{
  \stepcounter{adresscounter}
  \ifnum\value{adresscounter}=1
    \g@addto@macro\@adressstringa{\ifnum\authortagsused=0\def\br{\\}\else\def\br{, }\fi\adresstag{#1}#2}
    \g@addto@macro\@adressstringb{\def\br{\\}\adresstag{#1}\parbox[t]{14cm}{#2}}
  \else
    \g@addto@macro\@adressstringa{\\[\bigskipamount]\adresstag{#1}#2}
    \g@addto@macro\@adressstringb{\\[\medskipamount]\adresstag{#1}\parbox[t]{14cm}{#2}}
  \fi}
\def\@adress{\ifnum\value{denseversion}=0\@adressstringa\else\@adressstringb\fi}

\def\preprint#1{\gdef\@preprint{#1}}
\def\@preprint{}
\def\keywords#1{\gdef\@keywords{#1}}
\def\@keywords{}
\def\msc#1{\gdef\@msc{#1}}
\def\@msc{}
\def\email#1{
   \gdef\@email{#1}
   \g@addto@macro\@authorstringc{ {\it (#1)}}}
\def\@email{}
\def\dedication#1{\gdef\@dedication{#1}}
\def\@dedication{}

\def\mybaselinestretch#1{\gdef\@mybaselinestretch{#1}}
\def\@mybaselinestretch{}

\def\refname{References}

\mybaselinestretch{1.2}
\renewcommand{\baselinestretch}{\@mybaselinestretch}

\def\denseversion{
  \setcounter{denseversion}{1}
  \newgeometry{left=3cm,right=3cm,top=3cm}
  \mybaselinestretch{1.1}
  \renewcommand{\baselinestretch}{\@mybaselinestretch}
  \normalfont
  \fancyfoot[C]{\itshape{\hspace{2.5cm}--$\,\,$\thepage$\,\,$--}}}

\newlength{\myparskip}
\newlength{\myproofparskip}

\renewcommand{\emph}[1]{\def\reserved@a{it}\ifx\f@shape\reserved@a\uline{#1}\else\textit{#1}\fi}

\newcommand{\mytableofcontents}{
   \ifnum\value{denseversion}=0
     \tableofcontents
   \else
     \renewcommand{\baselinestretch}{0.8}
     \normalfont
     \tableofcontents
     \renewcommand{\baselinestretch}{\@mybaselinestretch}
     \normalfont
   \fi}

\newlength{\zeilenlaenge}
\def\putindent#1{
  \settowidth{\zeilenlaenge}{#1}
  \ifnum\zeilenlaenge>\textwidth
    #1
  \else
    \noindent #1
  \fi
}

\def\href#1#2{#2}

\def\kohyp{
  \usepackage{hyperref}
  \hypersetup{
    linktocpage = true,
    pdftitle = {\@title},
    pdfauthor = {\@author},
    pdfkeywords = {\@keywords},    
    bookmarksopen = true,
    bookmarksopenlevel = 1
  }}  
\def\showkeywords{\begin{flushleft}\footnotesize\textbf{Keywords}: \@keywords\end{flushleft}}
\def\showmsc{\begin{flushleft}\footnotesize\textbf{MSC 2010}: \@msc\end{flushleft}}

\newcounter{mythm}[subsection]
\newcounter{mainthm}

\def\setsecnumdepth#1{
  \setcounter{secnumdepth}{#1}
  \setcounter{mythm}{0}
  \ifnum \c@secnumdepth >0
    \ifnum \c@secnumdepth >1
      \def\themythm{\thesubsection.\arabic{mythm}}
      \numberwithin{equation}{subsection}
      \renewcommand\theequation{\thesubsection.\arabic{equation}}
    \else
      \def\themythm{\thesection.\arabic{mythm}}
      \numberwithin{equation}{section}
      \renewcommand\theequation{\thesection.\arabic{equation}}
    \fi
  \else
    \def\themythm{\arabic{mythm}}
  \fi}

\newenvironment{mythmenv}{\strut\ \setlength{\parskip}{\myproofparskip}}{\setlength{\parskip}{\myparskip}}

\newlength{\mythmskip}
\newlength{\mythmtopskip}
\newtheoremstyle{mythmstylea}{\mythmtopskip}{\mythmskip}{\it}{}{\bf}{.}{0em}{}
\newtheoremstyle{mythmstyleb}{\mythmtopskip}{\mythmskip}{}{}{\bf}{.}{0em}{}

\theoremstyle{mythmstylea}
\newtheorem{mytheorem}[mythm]{Theorem}
\newtheorem{mydefinition}[mythm]{Definition}
\newtheorem{mycorollary}[mythm]{Corollary}
\newtheorem{myproposition}[mythm]{Proposition}
\newtheorem{mylemma}[mythm]{Lemma}

\newtheorem{mymaintheorem}[mainthm]{Theorem}
\newtheorem{mymaincorollary}[mainthm]{Corollary}
\newtheorem{mymainproposition}[mainthm]{Proposition}
\newtheorem{mymaindefinition}[mainthm]{Definition}

\theoremstyle{mythmstyleb}
\newtheorem{myremark}[mythm]{Remark}
\newtheorem{myexample}[mythm]{Example}
\newtheorem{myexercise}[mythm]{Exercise}

\newenvironment{theorem}[1][]{\begin{mytheorem}[#1]\begin{mythmenv}}{\end{mythmenv}\end{mytheorem}}
\newenvironment{definition}[1][]{\begin{mydefinition}[#1]\begin{mythmenv}}{\end{mythmenv}\end{mydefinition}}
\newenvironment{corollary}[1][]{\begin{mycorollary}[#1]\begin{mythmenv}}{\end{mythmenv}\end{mycorollary}}
\newenvironment{proposition}[1][]{\begin{myproposition}[#1]\begin{mythmenv}}{\end{mythmenv}\end{myproposition}}
\newenvironment{lemma}[1][]{\begin{mylemma}[#1]\begin{mythmenv}}{\end{mythmenv}\end{mylemma}}
\newenvironment{remark}[1][]{\begin{myremark}[#1]\begin{mythmenv}}{\end{mythmenv}\end{myremark}}
\newenvironment{example}[1][]{\begin{myexample}[#1]\begin{mythmenv}}{\end{mythmenv}\end{myexample}}

\renewenvironment{proof}[1][Proof]{\noindent #1. \begin{mythmenv}}{\hfill$\square$\end{mythmenv}\medskip}

\def\mytitle{}
\def\zmptitle{
  \begin{tabular}{cc}
    \begin{minipage}[c]{0.4\textwidth}
      \begin{flushleft}
        \includegraphics[width=110pt]{../../tex/zmp}
      \end{flushleft}  
    \end{minipage}&
    \begin{minipage}[c]{0.55\textwidth}
      \begin{flushright}
      {\small\sf\@preprint}
      \end{flushright}
    \end{minipage}
  \end{tabular}
  \vskip 2cm}

\def\maketitle{
  \setlength{\parskip}{\myparskip}  
  \newpage
  \noindent
  \mytitle
  \begin{center}
    \LARGE\@title\\
    \if!\@subtitle!\else \smallskip\LARGE\@subtitle\\\fi
    \bigskip
    \if!\@author!\else\bigskip\large\@author\\\fi
    \ifnum\value{denseversion}=0
      \if!\@adress!\else     \bigskip\normalsize\@adress\\\fi
      \if!\@email!\else\ifnum\value{authorcounter}=1\bigskip\normalsize\textit{\@email}\\\else\fi\fi
    \else
    \fi
    \if!\@dedication!\else \bigskip\normalsize{\@dedication}\\\fi
  \end{center}
  \ifnum\value{denseversion}=0\vskip 1.5cm\else\vskip0.5cm\fi
  \thispagestyle{empty}}

\def\kobiburl#1{
   \IfBeginWith
     {#1}
     {http://arxiv.org/abs/}
     {\kobibarxiv{#1}}
     {\kobiblink{#1}}}
\def\kobibarxiv#1{\href{#1}{\texttt{[arxiv:\StrGobbleLeft{#1}{21}]}}}
\def\kobiblink#1{Available as: \href{#1}{\texttt{\StrSubstitute{#1}{_}{\underline{\;\;}}}}}

\def\kobib#1{
  \begin{raggedright}
  \ifnum\value{denseversion}=0\else\small\fi

  \end{raggedright}
  \ifnum\value{denseversion}=0\else
      \noindent
      \if!\@authorstringc!\else
        \ifnum\authortagsused=0\ifnum\value{authorcounter}>1\normalsize\@authorstringc\\[\medskipamount]\else\fi\else\normalsize\@authorstringc\\[\medskipamount]\fi
      \fi
      \if!\@adress!\else\normalsize\@adress\\{}\fi
      \ifnum\authortagsused=0
        \ifnum\value{authorcounter}=1
          \if!\@email!\else\linebreak\normalsize\textit{\@email}\\{}\fi
        \else
        \fi
      \else
      \fi
  \fi
  }

\newenvironment{commentfigure}{}
\newenvironment{sidewayscommentfigure}{\begin{minipage}}{\end{minipage}}

\def\showcomments{ -- Comments suppressed}

\newif\if@fewtab\@fewtabtrue{
  \count255=\time\divide\count255 by 60
  \xdef\hourmin{\number\count255}
  \multiply\count255 by-60\advance\count255 by\time
  \xdef\hourmin{\hourmin:\ifnum\count255<10 0\fi\the\count255}}
\def\ps@draft{
  \let\@mkboth\@gobbletwo
  \def\@oddfoot{
    \hbox to 7 cm{\tiny \versionno\hfil}
    \hskip -7cm\hfil\rm\thepage\hfil{\tiny\draftdate}}
  \def\@oddhead{}
  \def\@evenhead{}
  \let\@evenfoot\@oddfoot}
\def\draftdate{\number\month/\number\day/\number\year\ \ \ \hourmin }
\newcommand\version[1]{
  \typeout{}\typeout{#1}\typeout{}
  \vskip-1.7cm \centerline{\fbox{{\normalsize\tt DRAFT -- #1 -- 
  \draftdate\showcomments}}} \vskip0.92cm}
\def\draft#1{
  \def\versionno{#1}
  \pagestyle{draft}\thispagestyle{draft}
  \gdef\@ntitle{\version\versionno \@title}
  \global\def\draftcontrol{1}}
\global\def\draftcontrol{0}

\makeatother

\usepackage{amssymb}
\usepackage{amsmath}
\usepackage{amstext}
\usepackage[]{geometry}
\usepackage{mathrsfs}
\usepackage{euscript}
\usepackage[all]{xy}
\usepackage{xstring}
\usepackage{slashed}

\def\C {\mathbb{C}}

\def\id{\mathrm{id}}

\def\trivlin{\mathbf{I}}

\def\quand{\quad\text{ and }\quad}

\def\subset{\subseteq}

\def\nobr{~\hspace{-0.26em}}
\def\maps{\nobr:\nobr}
\def\df{\nobr := \nobr}

\let\Oldin\in\renewcommand{\in}{\nobr\Oldin\nobr}
\let\Oldtimes\times\renewcommand{\times}{\nobr\Oldtimes}
\let\Oldotimes\otimes\renewcommand{\otimes}{\nobr\Oldotimes}

\newlength{\widthtmp}
\def\length#1{\settowidth{\widthtmp}{#1}\the\widthtmp}

\renewcommand{\varepsilon}{\epsilon}


\def\erf#1{(\ref{#1})}

\newlength{\myl}

\def\brackets#1{\IfStrEq{#1}{-}{}{(#1)}}
\def\subindex#1{\IfStrEq{#1}{-}{}{_{#1}}}

\newcommand{\alxy}[1]{\begin{aligned}\xymatrix{#1}\end{aligned}}
\newcommand{\alxydim}[2]{\begin{aligned}\xymatrix#1{#2}\end{aligned}}

\renewcommand{\to}{\nobr\!\xymatrix@R=0cm@C=1.4em{\ar[r] &}\nobr}
\renewcommand{\mapsto}{\!\xymatrix@R=0cm@C=1.4em{\ar@{|->}[r] &}\!}
\renewcommand{\Rightarrow}{\!\xymatrix@R=0cm@C=1.4em{\ar@{=>}[r] &}\!}
\renewcommand{\Leftarrow}{\!\xymatrix@R=0cm@C=1.4em{\ar@{<=}[r] &}\!}
\newcommand{\incl}{\!\xymatrix@R=0cm@C=1.4em{\ar@{^(->}[r] &}\!}
\renewcommand{\hookrightarrow}{\incl}
\renewcommand\Leftrightarrow{\!\xymatrix@R=0cm@C=1.4em{\ar@{<=>}[r] &}\!}

\usepackage[latin1]{inputenc}
\usepackage[english]{babel}

\def\quot#1{``#1''}

\def\tetraeder{tetrahedron\ }

\def\zickzack{zigzag\ }
\def\codescent{codescent 2-groupoid}

\def\unifier{unifier}
\def\pt{pseudonatural transformation}
\def\pe{pseudonatural equivalence}

\newcommand{\loctrivfunct}[3]{\mathrm{Triv}^{#3}_{\pi}(#1)}

\newcommand{\trans}[3]{\mathfrak{Des}_{\pi}^{#3}(#1)}
\newcommand{\ex}[1]{\mathrm{Ex}_{#1}}

\newcommand{\upp}[3]{\mathcal{P}^{#1}_{#2}(#3)}
\newcommand{\uppp}{p^{\pi}}

\newcommand{\holo}{\mathscr{F}}

\newcommand{\bigon}[5]{\alxydim{@C=1.2cm}{#1 \ar@/^1.5pc/[r]^{#3}="1" \ar@/_1.5pc/[r]_{#4}="2" \ar@{=>}"1";"2"|{#5} & #2}}

\newcommand{\quadrat}[9]{\alxydim{@C=1.2cm@R=1.2cm}{#1 \ar[r]^{#5} \ar[d]_{#6} & #2 \ar@{=>}[dl]|{#9} \ar[d]^{#7} \\ #3 \ar[r]_{#8} & #4}}

\newcommand{\dreieck}[7]{\alxydim{@R=1.2cm@C=0.3cm}{& #2  \ar[dr]^{#5} & \\ #1 \ar[ur]^{#4} \ar[rr]_{#6}="1" && #3 \ar@{=>}[ul];"1"|{#7}}
}

\newcommand{\kopfdreieck}[7]{\alxydim{@R=1.2cm@C=0.3cm}{#1 \ar[dr]_{#4} \ar[rr]^{#6}="1" && #3 \ar@{=>}"1";[dl]|{#1} \\ & #2  \ar[ur]_{#5} & }
}

\title{Local Theory for 2-Functors on Path 2-Groupoids}

\author[a]{Urs Schreiber}	
\email{urs.schreiber@gmail.com}
\adress[a]{Faculty of Science\br Radboud University Nijmegen\\ Heyendaalseweg 135\br 6525 AJ Nijmegen\br The Netherlands}

\author[b]{Konrad Waldorf}
\email{konrad.waldorf@mathematik.uni-regensburg.de}
\adress[b]{Fakultät für Mathematik\br Universität Regensburg\\ Universitätsstraße 31\br 93053 Regensburg\br Germany}


\kohyp

\usepackage{amstext}
\usepackage{amsmath}

\begin{document}


\maketitle

\begin{abstract}
In this article we discuss local aspects of 2-functors defined  on the path 2-groupoid of a smooth manifold; in particular, local trivializations and  descent data. This is a contribution to a project that provides  an axiomatic formulation of connections on (possibly non-abelian) gerbes in terms of 2-functors. The main result of this paper establishes the first part of this formulation: we prove an equivalence between the globally defined 2-functors and their locally defined descent data. The second part  appears in a separate publication; there we prove equivalences between descent data, on one side, and various existing versions of gerbes with connection on the other side.   

\end{abstract}



\setsecnumdepth{1}

\tableofcontents

\setsecnumdepth{1}

\section{Introduction}

The present article is a contribution to the axiomatic formulation of connections on (possibly non-abelian) gerbes carried out by the authors in several articles \cite{schreiber3,schreiber5,schreiber2}. It  is based on 2-functors
\begin{equation*}
F \maps  \mathcal{P}_2(M)  \to T
\end{equation*}
defined on the path 2-groupoid of a smooth manifold $M$, with values in some \quot{target} 2-category $T$. The path 2-groupoid $\mathcal{P}_2(M)$ is a strict 2-groupoid with objects the points of $M$, 1-morphisms certain homotopy classes of paths, and 2-morphisms certain homotopy classes of homotopies between paths.  A typical example of a target 2-category is the 2-category of algebras (over some fixed field), bimodules, and  intertwiners. 
In that example, the algebra $F(x) \in T$ associated to  a point $x\in M$ is supposed to be the \emph{fibre} of the gerbe at $x$, the bimodule
\begin{equation*}
F(\gamma) \maps  F(x) \to F(y)
\end{equation*}
associated to a path $\gamma$ from $x$ to $y$ is supposed to be the \emph{parallel transport} of the connection on that gerbe \emph{along the path $\gamma$}, and the intertwiner
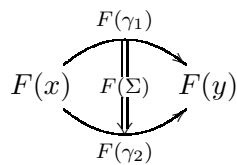
\begin{equation*}
\bigon{F(x)}{F(y)}{F(\gamma_1)}{F(\gamma_2)}{F(\Sigma)}
\end{equation*}
associated to a homotopy $\Sigma$ between paths $\gamma_1,\gamma_2$ from $x$ to $y$ is supposed to be the \emph{parallel transport} of the connection on that gerbe \emph{along the surface $\Sigma$}.

In this article we discuss local properties of 2-functors defined on path 2-groupoids. Our main objective is  to implement  the notions of \emph{structure  2-groupoids}, \emph{local trivializations}, and \emph{descent data} for a 2-functor $F$. Descent data plays a crucial role for specifying certain smoothness conditions for a 2-functor $F$.

In Section \ref{sec:loctrivdesc} we set up the basis for our discussion of local properties. In Section \ref{sec:loctriv}
 we introduce the notion of a \emph{local trivialization} of a 2-functor $F\maps \mathcal{P}_2(M) \to T$ (Definition \ref{sec2_def5}), which is central for this article. In order to say what a local trivializations is, we fix a strict 2-groupoid $\mathrm{Gr}$ and  a 2-functor $i \maps  \mathrm{Gr} \to T$. Basically, a 2-functor $F$ is locally $i$-trivializable if it factors locally through the 2-functor $i$. In more detail, we require that there exists a surjective submersion $\pi \maps Y \to M$, a strict 2-functor $\mathrm{triv} \maps  \mathcal{P}_2(Y) \to \mathrm{Gr}$, and a natural equivalence
\begin{equation}
\label{eq:nateq}
F \circ \pi_{*} \cong i \circ \mathrm{triv}\text{,}
\end{equation}
where $\pi_{*} \maps  \mathcal{P}_2(Y) \to \mathcal{P}_2(M)$ is the induced 2-functor on path 2-groupoids.   The 2-functor $i \maps  \mathrm{Gr} \to T$ plays the role of the typical fibre for $F$. Locally $i$-trivializable 2-functors form a 2-category $\mathrm{Funct}_{i}(\mathcal{P}_2(M),T)$.

In Section \ref{sec:descentdata} we introduce the notion of \emph{descent data} for locally trivialized 2-functors. The descent data of a local trivialization consists of the 2-functor $\mathrm{triv} \maps  \mathcal{P}_2(Y) \to \mathrm{Gr}$ and further coherence data related to the natural equivalence \erf{eq:nateq}. Descent data with respect to the structure 2-groupoid $i \maps  \mathrm{Gr} \to T$ forms a 2-category $\mathfrak{Des}^2(i)_M$. In Section \ref{sec:extract} we describe how to extract descent data from a given local trivialization of a 2-functor $F$.

Descent data plays a crucial role because it allows to impose \emph{smoothness conditions}. We shall briefly outline these conditions -- the full discussion is given in \cite{schreiber2}. We infer that the path 2-groupoid is a 2-groupoid internal to the category of \emph{diffeological spaces} \cite{schreiber5}. Diffeological spaces contain smooth manifolds as a full subcategory, but allow for many constructions which are obstructed for smooth manifolds. For example, the sets of 1-morphisms and 2-morphisms of the path 2-groupoid are quotients of subsets of mapping spaces between smooth manifolds, and all these operations lead to well-defined diffeological spaces. On the other side, however, the target 2-categories are typically not internal to the category of smooth manifolds or diffeological spaces, so that it is not possible to demand that $F$ is smooth. Instead, we will assume that the structure 2-groupoid $\mathrm{Gr}$ is a \emph{Lie 2-groupoid}, and impose smoothness conditions for the descent data of $F$, which is formulated with respect to $\mathrm{Gr}$.  For instance, the first  of these smoothness conditions is that the 2-functor $\mathrm{triv}\maps \mathcal{P}_2(Y) \to \mathrm{Gr}$  is smooth. There are further conditions related to the natural equivalence \erf{eq:nateq}; these will be described in detail in \cite{schreiber2}.

In Section \ref{sec:recon} we establish a procedure that \emph{reconstructs} a 2-functor from given descent data. This reconstruction procedure is, on a technical level, the main contribution of the present article. In order to explain some of the details, let $\mathfrak{Des}^2_{\pi}(i)\subset \mathfrak{Des}^2(i)_M$ denote the 2-category of descent data with respect to a fixed surjective submersion $\pi \maps Y \to M$. We introduce in Section \ref{sec2_1} the \emph{codescent 2-groupoid} $\mathcal{P}_2^{\pi}(M)$ whose idea is to combine the path 2-groupoid of $Y$ with additional \quot{vertical jumps} in the fibres of $\pi$.  The codescent 2-groupoid $\mathcal{P}^{\pi}_2(M)$ serves two purposes. Firstly, we prove in Section \ref{sec2_2} the existence of a 2-functor
\begin{equation*}
s \maps  \mathcal{P}_2(M) \to \mathcal{P}_2^{\pi}(M)
\end{equation*}
that consistently lifts points, paths, and homotopies between paths
from $M$ to $Y$, by compensating the differences between local lifts with the vertical jumps.  Secondly, we construct in Section \ref{sec2_3} a \quot{pairing 2-functor}
\begin{equation*}
R \maps  \mathfrak{Des}^2_{\pi}(i)  \to \mathrm{Funct}(\mathcal{P}_2^{\pi}(M),T)
\end{equation*}
expressing the result that the codescent 2-groupoid $\mathcal{P}_2^{\pi}(M)$ is \quot{$T$-dual} to the descent 2-category. The composition of a 2-functor in the image of $R$ with $s$ results in a locally $i$-trivializable 2-functor on $\mathcal{P}_2(M)$, and this defines the reconstruction of a 2-functor from descent data.

In Section \ref{sec:equivalence} we prove the main result of this article, namely that the correspondence between \emph{globally defined} locally $i$-trivializable 2-functors $F \maps  \mathcal{P}_2(M) \to T$ and \emph{locally defined} descent data is one-to-one. More precisely, we prove (Theorem \ref{th:main}) that extraction and reconstruction establish an equivalence
\begin{equation}
\label{eq:eq}
\mathfrak{Des}^2(i)_M \cong \mathrm{Funct}_{i}(\mathcal{P}_2(M),T)
\end{equation} 
between the 2-category $\mathfrak{Des}^2(i)_M$ of descent data with respect to the structure 2-groupoid $i \maps \mathrm{Gr} \to T$ and the 2-category $\mathrm{Funct}_{i}(\mathcal{P}_2(M),T)$ of locally $i$-trivializable 2-functors.

In order to explain the importance of this result let us return to the anticipated discussion of smoothness conditions that we impose on the descent data. Let us denote by $\mathfrak{Des}^2(i)^{\infty}_M$ the sub-2-category of $\mathfrak{Des}^2(i)_M$ that consists of \emph{smooth descent data}. The equivalence \erf{eq:eq} from the main result of the present article restricts to an equivalence between $\mathfrak{Des}^2(i)^{\infty}_M$ and a sub-2-category of $\mathrm{Funct}_{i}(\mathcal{P}_2(M),T)$, called the 2-category of \emph{transport 2-functors on $M$ with $\mathrm{Gr}$-structure}. These transport 2-functors constitute our axiomatic formulation of connections on gerbes. In this interpretation, our main theorem  implies an equivalence between transport 2-functors and local, smooth, differential-geometric data.

We have included an appendix containing a brief summary of the notions and conventions from higher category theory that we use.

\paragraph{Acknowledgements.}

We would like to thank John Baez for many discussions and suggestions. We thank the Hausdorff Research Institute for Mathematics in Bonn for kind hospitality and support.

\setsecnumdepth{2}

\section{Locally Trivial 2-Functors and their Descent Data}

\label{sec:loctrivdesc}

In this section we introduce the central notions of the local theory of 2-functors.

\subsection{Local Trivializations}

\label{sec:loctriv}

Let $M$ be a smooth manifold. For points $x,y\in M$, a \emph{path} $\gamma \maps x \to y$ is a smooth map $\gamma \maps [0,1] \to M$ with $\gamma(0)=x$ and $\gamma(1)=y$.  We require paths to have \quot{sitting instants}, i.e. to be locally constant around $\left \lbrace 0,1 \right \rbrace$. A \emph{bigon} $\Sigma \maps  \gamma \Rightarrow \gamma'$ between two paths $\gamma,\gamma' \maps x \to y$ is a smooth fixed-ends homotopy from $\gamma$ to $\gamma'$ with sitting instants at $\gamma$ and $\gamma'$.
The \emph{path 2-groupoid} $\mathcal{P}_2(M)$ \cite{schreiber5} of a smooth manifold $M$  is a strict 2-groupoid with
\begin{enumerate}[(i)]
\item 
objects: the points of $M$.

\item
1-morphisms: rank-one homotopy classes of  paths.

\item
2-morphisms: rank-two homotopy classes of bigons.
\end{enumerate}
The process of taking classes by homotopies of certain rank is explained in \cite{schreiber5}. For the purpose of this article, it suffices to accept that these assure the existence of strict, associative compositions and of strict inverses. For  definitions and conventions related to 2-categories we refer to Appendix \ref{app1}. 
If $f \maps M \to N$ is a smooth map, we get a 2-functor $f_{*} \maps \mathcal{P}_2(M) \to \mathcal{P}_2(N)$. For composable smooth maps $f$ and $g \maps N \to O$ we get
\begin{equation}
\label{eq:pullbackcomp}
(g \circ f)_{*} = g_{*} \circ f_{*}\text{.}
\end{equation}

In this article, we study  2-functors
\begin{equation}
\label{eq:2functors}
F \maps  \mathcal{P}_2(M) \to T
\end{equation}
for $T$ some 2-category called  \emph{target 2-category}. The 2-category of 2-functors \erf{eq:2functors} is denoted $\mathrm{Funct}(\mathcal{P}_2(M),T)$. If $f \maps M \to N$ is a smooth map, the composition of a 2-functor $F \maps  \mathcal{P}_2(N) \to T$ with the 2-functor $f_{*} \maps  \mathcal{P}_2(M) \to \mathcal{P}_2(N)$ is denoted by
\begin{equation*}
f^{*}F := F \circ f_{*} \maps  \mathcal{P}_2(M) \to T\text{.}
\end{equation*}
Local trivializations of a 2-functor $F \maps  \mathcal{P}_2(M) \to T$ 
have three attributes:
\begin{enumerate}[(i)]
\item
A  strict 
2-groupoid $\mathrm{Gr}$, the \emph{structure 2-groupoid}. 

\item
A   2-functor $i \maps \mathrm{Gr} \to T$ that indicates how the structure 2-groupoid is realized in the target 2-category. 

\item 
A surjective
submersion $\pi \maps Y \to M$ implementing locality.

\end{enumerate}
For a surjective submersion $\pi \maps Y \to M$ the fibre products $Y^{[k]} \df Y \times_M ... \times_M Y$ are again smooth manifolds, and the canonical projections
$\pi_{i_1...i_p} \maps Y^{[k]} \to Y^{[p]}$
to the indexed factors are smooth maps. 

\begin{definition}
\label{sec2_def5}
A \emph{$\pi$-local $i$-trivialization} of a 2-functor 
\begin{equation*}
F \maps  \mathcal{P}_2(M)
\to T
\end{equation*}
is a pair $(\mathrm{triv},t)$ of a strict 2-functor $\mathrm{triv} \maps  \mathcal{P}_2(Y) \to \mathrm{Gr}$ and of
a \pe
\begin{equation*}
\quadrat{\mathcal{P}_2(Y)}{\mathcal{P}_2(M)}{\mathrm{Gr}}{T\text{.}}{\pi_{*}}{\mathrm{triv}}{F}{i}{t}
\end{equation*}
\end{definition}

In other words, a 2-functor $F$ is locally trivializable, if its pullback $\pi^{*}F$ to the covering space factorizes -- up to  \pe\   -- through the structure 2-groupoid $\mathrm{Gr}$. A 2-functor is called \emph{$i$-trivializable}, if it has a $\id_M$-local $i$-trivialization.  A 2-functor $\mathrm{triv} \maps  \mathcal{P}_2(M) \to \mathrm{Gr}$ is called $i$-trivial, in which case we write $\mathrm{triv}_i  \df  i \circ \mathrm{triv}$  in order to abbreviate the notation. We also remark that by \quot{pseudonatural equivalence} we mean a pseudonatural transformation \emph{together} with  a weak inverse and  two invertible modifications expressing the invertibility (see Appendix \ref{app1}).

We define 
a 2-category $\loctrivfunct{i}{\mathrm{Gr}}{2}$ of  2-functors with  $\pi$-local $i$-trivialization: an object is a triple $(F,\mathrm{triv},t)$ of a 2-functor $F \maps \mathcal{P}_2(M) \to T$ together with a  $\pi$-local $i$-trivialization $(\mathrm{triv},t)$. A 1-morphism
\begin{equation*}
(F,\mathrm{triv},t) \to (F',\mathrm{triv}',t')
\end{equation*}
is just a \pt\ $F \to F'$ between the two 2-functors (ignoring the trivialization), and a 2-morphism is just a modification between those.

\subsection{Descent Data}

\label{sec:des}
\label{sec:descentdata}

Let $i \maps  \mathrm{Gr} \to T$ be a 2-functor from a strict 2-groupoid $\mathrm{Gr}$ to a 2-category $T$, and let $\pi \maps Y \to M$ be a surjective submersion. 
In the following three definitions, we define a 2-category $\trans{i}{\pi}{2}$ of \emph{descent data}. 

\begin{definition}
\label{def:descobj}
A \uline{descent object}  is a  tuple $(\mathrm{triv},g,\psi,f)$  consisting of 
\begin{enumerate}[(i)]
\item 
a strict  2-functor
$\mathrm{triv} \maps \mathcal{P}_2(Y) \to \mathrm{Gr}$
\item
a \pe\
$g \maps  \pi_1^{*}\mathrm{triv}_i \to \pi_2^{*}\mathrm{triv}_i$

\item
an invertible modification $\psi \maps  \id_{\mathrm{triv}_i} \Rightarrow \Delta^{*}g$

\item
an invertible modification
$f \maps \pi_{23}^{*}g \circ \pi_{12}^{*}g \Rightarrow \pi_{13}^{*}g$
\end{enumerate}
such that the diagrams  
\begin{equation}
\label{2}
\alxydim{@R=1.1cm@C=0.6cm}{\id_{\pi_2^{*}\mathrm{triv}_i} \circ g \ar@{=>}[rr]^{\pi_2^{*}\psi \circ \id} \ar@{=>}[dr]_{r} && \Delta_{22}^{*}g \circ g\ar@{=>}[dl]^{\Delta_{122}^{*}f} \\ &g&}
\quad\text{, }\quad
\alxydim{@R=1.1cm@C=0.6cm}{g \circ \id_{\pi_1^{*}\mathrm{triv}_i} \ar@{=>}[rr]^{\id \circ \pi_1^{*}\psi} \ar@{=>}[dr]_{l} && g \circ \Delta_{11}^{*}g\ar@{=>}[dl]^{\Delta_{112}^{*}f} \\ &g&}
\end{equation}
and 
\begin{equation}
\label{3}
\alxydim{@C=-0.8cm@R=1.3cm}{&&& (\pi_{34}^{*}g \circ \pi_{23}^{*}g) \circ \pi_{12}^{*}g \ar@{=>}[dlll]_{a} \ar@{=>}[drrr]^{\pi_{234}^{*}f \circ \id} &&& \\ \pi_{34}^{*}g \circ (\pi_{23}^{*}g \circ \pi_{12}^{*}g) \ar@{=>}[drr]_{\id \circ \pi_{123}^{*}f} &&&&&& \hspace{1.7cm}\pi_{24}^{*}g \circ \pi_{12}^{*}g\hspace{1.7cm} \ar@{=>}[dll]^{\pi_{124}^{*}f} \\ && \pi_{34}^{*}g \circ \pi_{13}^{*}g \ar@{=>}[rr]_-{\id \circ \pi_{134}^{*}f} && \pi_{14}^{*}g\text{.} &&}\hspace{-2cm}
\end{equation}
are commutative. A descent object $(\mathrm{triv},g,\psi,f)$ is called \emph{normalized}, if the conditions
\begin{equation*}
\id_{\mathrm{triv}_i} = \Delta^{*}g
\quand
g \circ \Delta_{21}^{*}g= \id_{\pi_1^{*}\mathrm{triv}_i}
\end{equation*}
and 
\begin{equation*}
\psi = \id_{\Delta^{*}g}
\quand
\Delta_{121}^{*}f = \id_{\Delta_{11}^{*}g}
\end{equation*}
hold. 
\end{definition}
In these diagrams, $r$, $l$ and $a$ are the right and left \unifier s and the associator of the 2-category $T$. Further is $\Delta \maps Y \to Y^{[2]}$  the diagonal map,  $\Delta_{112},\Delta_{122} \maps Y^{[2]} \to Y^{[3]}$  are the maps that duplicate the first and the second factor, respectively, $\Delta_{jj} := \Delta \circ \pi_j$, and $\Delta_{21}:Y^{[2]} \to Y^{[2]}$ exchanges the two components. Normalized descent objects play an important role in the discussion of surface holonomy; see Lemma \ref{reconinversionstrict} and  \cite[Section 5]{schreiber2}.

\begin{definition}
Let $(\mathrm{triv},g,\psi,f)$ and $(\mathrm{triv}',g',\psi',f')$ be descent objects. A \emph{descent 1-morphism} $(\mathrm{triv},g,\psi ,f) \to (\mathrm{triv}',g',\psi ',f')$
  is a pair $(h,\varepsilon)$  of a 
  \pt
\begin{equation*}
h \maps  \mathrm{triv}_i \to \mathrm{triv}_i'
\end{equation*}
and  an invertible    modification
\begin{equation*}
\epsilon \maps  \pi_2^{*}h \circ g \Rightarrow g' \circ \pi_1^{*}h
\end{equation*}
such that the diagrams
\begin{equation}
\label{5}
\alxydim{@C=2cm}{\pi_{23}^{*}g' \circ (\pi_{2}^{*}h \circ \pi_{12}^{*}g) \ar@{=>}[r]^{a} \ar@{=>}[d]_{\id \circ \pi_{12}^{*}\epsilon} & (\pi_{23}^{*}g' \circ \pi_{2}^{*}h) \circ \pi_{12}^{*}g \ar@{=>}[d]^{\pi_{23}^{*}\epsilon^{-1} \circ \id}   \ar@{=>}[d]\\\pi_{23}^{*}g' \circ (\pi_{12}^{*}g' \circ \pi_1^{*}h)  \ar@{=>}[d]_{a^{-1}} & (\pi_{3}^{*}h \circ \pi_{23}^{*}g) \circ \pi_{12}^{*}g \ar@{=>}[d]^{a}   \ar@{=>}[d] \\(\pi_{23}^{*}g' \circ \pi_{12}^{*}g') \circ \pi_1^{*}h\ar@{=>}[d]_{f' \circ \id} &\pi_{3}^{*}h \circ (\pi_{23}^{*}g \circ \pi_{12}^{*}g)\ar@{=>}[d]^{\id\circ f} \\ \pi_{13}^{*}g'\circ \pi_1^{*}h \ar@{=>}[r]_{\pi_{13}^{*}\epsilon} & \pi_{3}^{*}h \circ \pi_{13}^{*}g\text{.}}
\end{equation}
and
\begin{equation}
\label{4}
\alxydim{}{\id_{\mathrm{triv}'_i} \circ h \ar@{=>}[r]^-{l_{h}} \ar@{=>}[d]_{\psi' \circ \id_h} & h \ar@{=>}[r]^-{r_{h}^{-1}} & h \circ \id_{\mathrm{triv}_i} \ar@{=>}[d]^{\id_h \circ \psi} \\ \Delta^{*}g' \circ h \ar@{=>}[rr]_{\Delta^{*}\epsilon} && h \circ \Delta^{*}g\text{.}}
\end{equation}
are commutative.
\end{definition}

Finally, we introduce

\begin{definition}
Let $(h_1,\varepsilon_1)$ and $(h_2,\varepsilon_2)$ be descent 1-morphisms from a descent object $(\mathrm{triv},g,\psi,f)$ to another descent object $(\mathrm{triv}',g',\psi',f')$. A \emph{descent 2-morphism} $(h_1,\varepsilon_1) \Rightarrow (h_2,\varepsilon_2)$ is a modification 
\begin{equation*}
E \maps  h_1 \Rightarrow h_2
\end{equation*}
such that the diagram
\begin{equation}
\label{6}
\alxydim{@=1.2cm}{g' \circ \pi_1^{*}h_1 \ar@{=>}[d]_{\id \circ \pi_1^{*}E}  \ar@{=>}[r]^{\varepsilon_{1}} & \pi_2^{*}h_1 \circ g \ar@{=>}[d]^{\pi_2^{*}E \circ \id } \\ g' \circ \pi_1^{*}h_2 \ar@{=>}[r]_{\varepsilon_2} & \pi_2^{*}h_2 \circ g\text{.}}
\end{equation}
is commutative.
\end{definition}

Descent objects, 1-morphisms and 2-morphisms form a  2-category $\trans{i}{\pi}{2}$ in an evident way. 
We remark that this  2-category comes with a strict 2-functor 
\begin{equation*}
V \maps  \trans{i}{\mathrm{Gr}}{2} \to \mathrm{Funct}(\mathcal{P}_2(Y), T)\text{.}
\end{equation*}
From a descent object $(\mathrm{triv},g,\psi,f)$ it keeps only the 2-functor $\mathrm{triv}$ and from a descent 1-morphism $(h,\varepsilon)$ only the \pt\ $h$. 

\begin{example}
Let us briefly consider the  2-category $\trans{i}{2}{2}$ for the particular case that the manifolds $M$ and $Y$ are just points. Let $\mathfrak{C}$ be a tensor category, and let $\mathcal{B}\mathfrak{C}$ be the 2-category with one object associated to $\mathfrak{C}$, see Example \ref{app2cat_ex2}.  Let $\mathrm{Gr}$ be the trivial 2-groupoid (one object, one 1-morphism and one 2-morphism), and let $i \maps  \mathrm{Gr} \to \mathcal{B}\mathfrak{C}$ be the  2-functor that sends the unique 1-morphism to the tensor unit in $\mathfrak{C}$. Then, a descent object is precisely  a \emph{special symmetric Frobenius algebra} object in $\mathfrak{C}$. 
\end{example}

\subsection{Extraction of Descent Data}

\label{sec1_3}
\label{sec:extract}

We have so far introduced a 2-category $\loctrivfunct{i}{\mathrm{Gr}}{2}$ of 2-functors with $\pi$-local $i$-trivializations and a  2-category $\trans{i}{\mathrm{Gr}}{2}$ of descent data, both associated to a surjective submersion $\pi$ and a 2-functor $i \maps \mathrm{Gr} \to T$. Now we define a 2-functor
\begin{equation*}
\ex{\pi} \maps \loctrivfunct{i}{\mathrm{Gr}}{2} \to \trans{i}{\mathrm{Gr}}{2}
\end{equation*}
between these 2-categories.
This 2-functor \emph{extracts} descent data from  2-functors with local trivializations.

Let $F \maps \mathcal{P}_2(M) \to
T$ be a  2-functor with a $\pi$-local $i$-trivialization $(\mathrm{triv},t)$. We recall that by our conventions the pseudonatural equivalence $t$ comes with a weak inverse $\bar t \maps \mathrm{triv}_i
\to \pi^{*}F$ and with invertible  modifications
\begin{equation}
\label{7}
i_t :  \bar t \circ t \Rightarrow \id_{\pi^{*}F}
\quad\text{ and }\quad
j_t :  \id_{\mathrm{triv}_i} \Rightarrow t \circ \bar t
\end{equation}
satisfying the identities (\ref{app2cat_1}). We define a \pe 
\begin{equation*}
g: \pi_1^{*}\mathrm{triv_i} \to \pi_2^{*}\mathrm{triv}_i
\end{equation*}
as the composition 
$g  \df  \pi_2^{*}t \circ \pi_1^{*}\bar t$ of pseudonatural equivalences. This composition is well-defined since $\pi_1^{*}\pi^{*}F=\pi_2^{*}\pi^{*}F$.   We obtain $\Delta^{*}g= t \circ \bar t$, so that the definition $\psi \df  j_t$ yields an invertible modification
\begin{equation*}
\psi \maps \id_{\mathrm{triv}_i} \Rightarrow \Delta^{*}g\text{.} 
\end{equation*}
Finally, we define an invertible modification 
\begin{equation*}
f \maps \pi_{23}^{*}g \circ \pi_{12}^{*}g \Rightarrow \pi_{13}^{*}g
\end{equation*}
as the composition
\begin{equation*}
\alxydim{@C=0.8cm}{(\pi_3^{*}t \circ \pi_2^{*}\bar t) \circ (\pi_2^{*}t \circ \pi_1^{*}\bar t) \ar@{=>}[r] & \pi_3^{*}t \circ (( \pi_2^{*}\bar t \circ \pi_2^{*}t )\circ \pi_1^{*}\bar t) \ar@{=>}[d]_{\id \circ (\pi_2^{*}i_t \circ \id)} \\ &\pi_3^{*}t \circ (\id_{\pi^{*}F} \circ \pi_1^{*}\bar t) \ar@{=>}[r]_-{\id\circ r_{\pi_1^{*}\bar t}}  &\pi_3^{*}t \circ \pi_1^{*}\bar t }
\end{equation*}
where  $r$ is the right \unifier\ of $\mathrm{Funct}(\mathcal{P}_2(Y^{[2]}),T)$, and the first arrow summarizes two obvious occurrences of  associators.

\begin{lemma}
\label{lem5}
The modifications $\psi$ and $f$  make the diagrams (\ref{2}) and (\ref{3}) commutative,
so that 
\begin{equation*}
\ex{\pi}(F,\mathrm{triv},t) \df (\mathrm{triv},g,\psi ,f)
\end{equation*}
 is a descent object.
\end{lemma}

\begin{proof}
We prove the commutativity of the diagram on the left hand side of (\ref{2}) by patching it together from commutative diagrams:
\begin{equation*}
\alxydim{@C=-0.05cm@R=0.6cm}{\id_{\pi_2^{*}\mathrm{triv}_i} \circ (\pi_2^{*}t \circ \pi_1^{*}\bar t)\ar@{}[dddrrrrr]|>>>>>>>>>>>>>>>>*+[F-:<8pt>]{C}\ar@{}[drrrrrrr]|*+[F-:<8pt>]{B} \ar@{=>}[ddddrrrrr]_{r} \ar@{=>}[drrrr]|-{a^{-1}} \ar@{=>}[rrrrrr]|-{j_t \circ (\id \circ \id)} &&&&&& \ast \ar@{}[ddrr]|*+[F-:<8pt>]{A} \ar@{=>}[d]|{a^{-1}} \ar@{=>}[rrr]|-
{a} &\hspace{2cm}&& \pi_2^{*}t \circ (\pi_2^{*}\bar t \circ (\pi_2^{*}t \circ \pi_1^{*}\bar t)) \ar@{=>}[ddl]|{\id \circ a^{-1}}  
\\
&&&& \ast \ar@{}[ddrr]|*+[F-:<8pt>]{D}\ar@{=>}[dddr]|{r \circ \id}  \ar@{=>}[rr]|{(j_t \circ\id) \circ \id} && \ast \ar@{=>}[d]|{a \circ \id} & 
\\
&&&&&&\ast \ar@{=>}[d]|{(\id \circ i_t)\circ \id} \ar@{=>}[rr]|{a}="1" \ar@{}"1";[dr]|*+[F-:<8pt>]{B} && \ast \ar@{=>}[dl]^{\id\circ(i_t \circ \id)}
\\
&&&&&&\ast \ar@{=>}[dl]|-{l\circ \id} \ar@{=>}[r]|{a} & \ast \ar@{=>}[dll]^{\id\circ r}
\\
&&&&&\pi_2^{*}t \circ \pi_1^{*}\bar t}
\end{equation*}
The six subdiagrams are commutative: A is the Pentagon axiom (C4) of $T$, B's are the naturality of the associator, C and D are diagrams that follow from the coherence theorem for the 2-category $T$, and the remaining small triangle is axiom (C2). The commutativity of the second diagram in (\ref{2}) and the one of diagram \erf{3} can be shown in the same way.
\end{proof}

Now let $A \maps F \to F'$ be a \pt\
between two 2-functors with $\pi$-local $i$-trivializations
$t \maps  \pi^{*}F \to \mathrm{triv}_i$ and $t' \maps  \pi^{*}F'
\to \mathrm{triv}'_i$. Let $i_t,j_t$ and $i_{t'},j_{t'}$ be the modifications (\ref{7}) we have chosen for the   weak inverses $\bar t$ and $\bar t'$.
We  define a \pt
\begin{equation*}
h \maps  \mathrm{triv}_i \to \mathrm{triv}'_i
\end{equation*}
by
$h  \df  (t' \circ \pi^{*}A) \circ \bar t$, and  an invertible modification $\varepsilon$
by
\begin{equation*}
\alxydim{@R=0.8cm}{\pi_2^{*}h \circ g \ar@{=>}[r] & (\pi_2^{*}t' \circ \pi_2^{*}\pi^{*}A) \circ ((\pi_2^{*} \bar t  \circ\pi_2^{*}t )\circ \pi_1^{*}\bar t) \ar@{=>}[d]^{(\pi_2^{*}l_{t'}^{-1}\circ \id) \circ (\pi_2^{*}i_t \circ \id)} \\ &((\pi_2^{*}t' \circ \id )\circ \pi_2^{*}\pi^{*}A) \circ (\id\circ \pi_1^{*}\bar t) \ar@{=>}[d]^{((\id \circ \pi_1^{*}i_{t'}^{-1}) \circ \id) \circ \pi_1^{*}r_t} \\ & ((\pi_2^{*}t' \circ (\pi_1^{*}\bar t' \circ \pi_1^{*}t') )\circ \pi_1^{*}\pi^{*}A) \circ \pi_1^{*}\bar t  \ar@{=>}[r] & g' \circ \pi_1^{*}h\text{.} }
\end{equation*}

Here, the unlabelled arrows summarize the definitions of $h$ and $g$ and several obvious occurrences of associators. Arguments similar to those given in the proof of Lemma \ref{lem5} show the following lemma.

\begin{lemma}
The modification $\varepsilon$  makes the diagrams (\ref{5}) and (\ref{4}) commutative, so that $\ex{\pi}(A) \df (h,\varepsilon)$ is  a descent 1-morphism
\begin{equation*}
\ex{\pi}(A)
 \maps  \ex{\pi}(F) \to \ex{\pi}(F').
\end{equation*}
\end{lemma}

In order to continue the definition of the 2-functor $\ex{\pi}$  we consider a modification $B \maps A_1 \Rightarrow A_2$
between \pt s $A_1,A_2 \maps F \to F'$ of  2-functors with $\pi$-local $i$-trivializations $t \maps  \pi^{*}F \to \mathrm{triv}_i$
and $t' \maps \pi^{*}F' \to \mathrm{triv}'_i$. Let $(h_k,\varepsilon_k)  \df  \ex{\pi}(A_k)$ be the associated descent 1-morphisms for $k=1,2$. We define a modification $E \maps  h_1 \Rightarrow h_2$
by
\begin{equation*}
\alxydim{@C=2cm}{h_1= (t' \circ \pi^{*}A_{1}) \circ \bar t \ar@{=>}[r]^{(\id \circ \pi^{*}B) \circ \id} & (t' \circ \pi^{*}A_2 ) \circ \bar t = h_2\text{.}}
\end{equation*}

\begin{lemma}
The modification $E$ makes the diagram (\ref{6}) commutative so that $\ex{\pi}(B)  \df  E$ is a descent 2-morphism
\begin{equation*}
\ex{\pi}(B) \maps  \ex{\pi}(A_1) \Rightarrow \ex{\pi}(A_2)\text{.}
\end{equation*}
\end{lemma}

In order to finish the definition of the 2-functor $\ex{\pi}$ we have to define its compositors and unitors.
We consider two composable \pt s $A_1 \maps F \to
F'$ and $A_2 \maps F' \to F''$ and the extracted descent 1-morphisms
$(h_k,\varepsilon_k)  \df  \ex{\pi}(A_k)$ for $k=1,2$ and $(\tilde h,\tilde\varepsilon) \df  \ex{\pi}(A_2 \circ A_1)$.
The compositor
\begin{equation*}
c_{A_1,A_2} \maps  \ex{\pi}(A_2) \circ \ex{\pi}(A_2)
\Rightarrow \ex{\pi}(A_2 \circ A_1)
\end{equation*}
is the modification $h_{2} \circ h_1 \Rightarrow \tilde h$ defined by
\begin{equation*}
\alxydim{@C=0.4cm}{((t'' \circ \pi^{*}A_2) \circ \bar t')\circ ((t' \circ \pi^{*}A_1) \circ \bar t) \ar@{=>}[r]  & (t'' \circ (\pi^{*}A_2 \circ (( \bar t' \circ t') \circ \pi^{*}A_1 ))) \circ \bar t \hspace{-2cm} \ar@{=>}[d]^{(\id \circ (\id \circ (i_{t'} \circ \id ))) \circ \id} \\ & \hspace{-2cm} (t'' \circ (\pi^{*}A_2 \circ (\id \circ \pi^{*}A_1 ))) \circ \bar t  \ar@{=>}[r] & (t'' \circ \pi^{*}(A_2 \circ A_1)) \circ \bar t\text{.}}
\end{equation*}
For a 2-functor $F \maps \mathcal{P}_2(M) \to T$ we find $\ex{\pi}(\id_{F})
= t \circ \bar t$. So, the unitor
\begin{equation*}
u_{F}  \maps \ex{\pi}(\id_{F})
\Rightarrow \id_{\mathrm{triv}_i}
\end{equation*}
is the modification $u_{F}  \df  j_t^{-1}$. The  identities (\ref{app2cat_1}) for $i_t$ and $j_t$ show that compositors and unitors are  descent 2-morphisms. The following statement is now straightforward to check. 

\begin{proposition}
The structure collected above furnishes a 2-functor
\begin{equation*}
\ex{\pi} \maps \loctrivfunct{i}{\mathrm{Gr}}{2} \to \trans{i}{\mathrm{Gr}}{2}\text{.}
\end{equation*}
\end{proposition}

\setsecnumdepth{2}

\section{Reconstruction from Descent Data}

\label{sec:recon}

We have so far described how  globally defined  2-functors induce locally defined structure, in terms of the 2-functor $\ex{\pi}$. In this section we describe a 2-functor going in the other direction. 

\subsection{A Covering of the Path 2-Groupoid}

\label{sec2_1}

In this section we introduce  the \emph{\codescent} $\upp{\pi}{2}{M}$ associated to a surjective submersion $\pi\maps Y \to M$. It combines the path 2-groupoid of $Y$
with additional jumps between  the fibres. This construction generalizes the groupoid $\upp{\pi}{1}{M}$ from \cite{schreiber3}.

The objects of $\upp{\pi}{2}{M}$ are all points $a \in Y$. There are  two \quot{basic} 1-morphisms:
\begin{enumerate}
\item[(1)]
\emph{Paths}: rank-one homotopy classes of paths $\gamma \maps a \to a'$ in $Y$.
\item[(2)]
\emph{Jumps}: points $\alpha \in Y^{[2]}$ considered
as  1-morphisms from $\pi_1(\alpha)$ to $\pi_2(\alpha)$.
\end{enumerate}
The set of 1-morphisms of $\upp{\pi}{2}{M}$ is freely generated from these two basic 1-morphisms, i.e. we have a formal composition $*$ and a formal identity $\id_a^{*}$ (the empty composition) associated to every object $a\in Y$. We introduce six \quot{basic} 2-morphisms:
\begin{enumerate}
\item[(1)]
Four 2-morphisms of  \emph{essential} type:
\begin{enumerate}
\item 

Rank-two homotopy classes of bigons  $\Sigma \maps \gamma_1 \Rightarrow \gamma_2$ in $Y$ going between paths.

\item
Rank-one homotopy classes of paths $\Theta \maps  \alpha \to \alpha'$ in $Y^{[2]}$ considered as  2-isomorphisms
\begin{equation*}
\Theta \maps  \alpha' * \pi_1(\Theta) \Rightarrow \pi_2(\Theta) * \alpha\text{,}
\end{equation*} 
going between  1-morphisms mixed from jumps and paths.

\item
Points $\Xi \in Y^{[3]}$ considered as 2-isomorphisms
\begin{equation*}
\Xi \maps \pi_{23}(\Xi) * \pi_{12}(\Xi) \Rightarrow \pi_{13}(\Xi)
\end{equation*}
going between jumps.
\item
Points $a\in Y$ considered as 2-isomorphisms 
\begin{equation*}
\Delta_a \maps  \id_a^{*} \Rightarrow (a,a)
\end{equation*}
relating the formal identity with the trivial jump.
\end{enumerate}
In (b) to (d) we demand that the 2-morphisms $\Theta$, $\Xi$ and $\Delta_a$ come with
formal inverses, denoted by $\Theta^{-1}$, $\Xi^{-1}$ and $\Delta_a^{-1}$.

\item[(2)] Two 2-morphisms of \emph{technical} type:
\begin{enumerate}
\item
associators for the formal composition, i.e. 2-isomorphisms
\begin{equation*}
a^{*}_{\beta_1,\beta_2,\beta_3}  \maps  (\beta_3 * \beta_2) * \beta_1 \Rightarrow \beta_3 * (\beta_2 * \beta_1)
\end{equation*}
for $\beta_k$ either paths or jumps,
and \unifier s
\begin{equation*}
l_{\beta} \maps  \beta * \id_{a}^{*} \Rightarrow \beta
\quad\text{ and }\quad
r_{\beta} \maps  \id_{b}^{*} * \beta \Rightarrow \beta\text{.}
\end{equation*} 
\item 
for points $a\in Y$ and composable paths $\gamma_1$ and $\gamma_2$ 2-isomorphisms 
\begin{equation*}
u^{*}_a \maps \id_a \Rightarrow \id_a^{*}
\quad\text{ and }\quad
c^{*}_{\gamma_1,\gamma_2} \maps \gamma_2 * \gamma_1 \Rightarrow \gamma_2 \circ \gamma_1
\end{equation*}
expressing that the formal composition restricted to paths yields  the usual composition of paths. 
\end{enumerate}
\end{enumerate}

Now we consider the set which is freely generated from these basic 2-morphisms in virtue of a formal horizonal composition $*$ and a formal vertical composition $\circledast$.  The formal identity 2-morphisms are denoted by  $\id_{\beta}^{\circledast} \maps \beta \Rightarrow \beta$ for any 1-morphism $\beta$.     
The set of 2-morphisms of the 2-category $\upp{\pi}{2}{M}$ is this set subject to the following list of identifications:
\begin{enumerate}
\item[(I)] Identifications of \emph{2-categorical type}. The formal compositions $*$ and $\circledast$, and the 2-isomorphisms of type (2a) form the structure of a 2-category and we impose all identifications required by   the axioms (C1) to (C4).

\item[(II)]
Identifications of \emph{2-functorial type}. We have the structure of a 2-functor
\begin{equation*}
\iota \maps  \mathcal{P}_2(Y) \to \upp{\pi}{2}{M}\text{.}
\end{equation*}
This 2-functor regards points, paths and bigons in $Y$ as objects, 1-morphisms of type (1) and 2-morphisms of type (1a), respectively. Its compositors and unitors are the 2-isomorphisms $c^{*}$ and $u^{*}$ of type (2b). We impose all identification required by the axioms (F1) to (F4) for this 2-functor.

\item[(III)]
Identifications of \emph{transformation type}. We have the structure of a \pt
\begin{equation*}
\Gamma \maps \pi_1^{*}\iota \to \pi_2^{*}\iota
\end{equation*}
between 2-functors defined over $Y^{[2]}$. Its component at a 1-morphism $\Theta \maps \alpha \to \alpha'$ in $\mathcal{P}_1(Y^{[2]})$ is the 2-isomorphism $\Theta$ of type (1b). We impose all identifications required by the axioms (T1) and (T2) for this \pt.

\item[(IV)]
Identification of \emph{modification type}. We have the structure of a modification
\begin{equation}
\label{40}
 \pi_{23}^{*}\Gamma \circ \pi_{12}^{*}\Gamma \Rightarrow \pi_{13}^{*}\Gamma
\end{equation}
between \pt s of 2-functors defined over $Y^{[3]}$. Its component at an object $\Xi\in Y^{[3]}$ is the 2-isomorphism $\Xi$ of type (1c). We have the structure of  another modification
\begin{equation}
\label{41}
\id_{\iota} \Rightarrow \Delta^{*}\Gamma
\end{equation}
between \pt s of 2-functors over $Y$, whose component at an object $a\in Y$ is the 2-isomorphism $\Delta_a$ of type (1d). 
We impose all identifications required by the commutativity of diagram (\ref{app2cat_2}) for both modifications. 
\item[(V)]
Identifications of \emph{essential type}: 
\begin{enumerate}
\item[1.]
For every point $\Psi\in Y^{[4]}$ we impose the commutativity of the diagram
\begin{equation*}
\alxydim{@C=-1.4cm@R=1.3cm}{&&& (\pi_{34}(\Psi) * \pi_{23}(\Psi)) * \pi_{12}(\Psi) \ar@{=>}[dlll]_{a^{*}} \ar@{=>}[drrr]^{\pi_{234}(\Psi)*\id^{*}} &&& \\ \pi_{34}(\Psi) * (\pi_{23}(\Psi) * \pi_{12}(\Psi)) \ar@{=>}[drr]_{\id^{*}*\pi_{123}(\Psi)} &&&&&& \hspace{2.2cm}\pi_{24}(\Psi) *\pi_{12}(\Psi)\hspace{2.2cm} \ar@{=>}[dll]^{\pi_{124}(\Psi)} \\ && \pi_{34}(\Psi) * \pi_{13}(\Psi) \ar@{=>}[rr]_-{\pi_{134}(\Psi)} && \pi_{14}(\Psi) &&}
\end{equation*}
of compositions of jumps.
\item[2.]
For every point $\alpha\in Y^{[2]}$ we impose the commutativity of the diagrams
\begin{equation*}
\alxydim{@R=1.1cm@C=0.6cm}{\id_{b}^{*} * \alpha\ \ar@{=>}[rr]^{b * \id_{\alpha}^{*}} \ar@{=>}[dr]_{r_{\alpha}^{*}} && (b,b) * \alpha\ar@{=>}[dl]^{(a,b,b)} \\ &\alpha&}
\text{ and }
\alxydim{@R=1.1cm@C=0.6cm}{\alpha * \id_{a}^{*} \ar@{=>}[rr]^{\id_{\alpha}^{*} \circ a} \ar@{=>}[dr]_{l_{\alpha}^{*}} && \alpha * (a,a)\ar@{=>}[dl]^{(a,a,b)} \\ &\alpha\text{.}&}
\end{equation*}
\end{enumerate}
\end{enumerate}

According to (I) we have defined a 2-category $\upp{\pi}{2}{M}$.  
\begin{lemma}
\label{app2pp_lem2}
The 2-category  $\upp{\pi}{2}{M}$ is a 2-groupoid. 
\end{lemma}

\begin{proof}
All 2-morphisms except those of type (1a) are invertible by definition. But for a 2-morphism of type (1a), a bigon $\Sigma \maps \gamma \Rightarrow \gamma'$, we have
\begin{equation*}
\Sigma^{-1} \circledast \Sigma \stackrel{\text{(II)}}{=} \Sigma^{-1} \bullet \Sigma = \id_{\gamma} \stackrel{\text{(II)}}{=} \id_{\gamma}^{\circledast}\text{,}
\end{equation*}
and analogously $\Sigma \circledast \Sigma^{-1}=\id_{\gamma'}^{\circledast}$. Here we have used identification (II); more precisely axiom (F1) of the 2-functor $\iota \maps  \mathcal{P}_2(Y) \to \upp{\pi}{2}{M}$. To see that a path $\gamma \maps a \to b$ is invertible, we claim that $\gamma^{-1}$ is a weak inverse. It is easy to construct the 2-isomorphisms $i_{\gamma}$ and $j_{\gamma}$ 
using the 2-isomorphisms of type (2b). The required identities (\ref{app2cat_1}) for these 2-isomorphisms are then satisfied due to  identification (II).
To see that a jump $\alpha\in Y^{[2]}$ with $\alpha=(x,y)$ is invertible, we claim that $\bar\alpha  \df  (y,x)$ is a weak inverse. The 2-isomorphisms $i_{\alpha}$ and $j_{\alpha}$ can be constructed from
2-isomorphisms of types (1c) and (1d). The identities (\ref{app2cat_1}) are satisfied du to identifications (V1) and (V2).
\end{proof}

We remark that we have  a 2-functor 
$\iota \maps  \mathcal{P}_2(Y) \hookrightarrow \upp{\pi}{2}{M}$,
a \pt\ $\Gamma$ and  modifications \erf{40} and \erf{41} claimed by identifications (II), (III) and (IV).

\subsection{The Section 2-Functor}

\label{sec2_2}
\label{app2}

There is  a canonical strict 2-functor
\begin{equation*}
\uppp \maps  \upp{\pi}{2}{M} \to \mathcal{P}_2(M)
\end{equation*}
whose composition with the 2-functor $\iota$ is equal to the 2-functor $\pi_{*} \maps  \mathcal{P}_2(Y) \to \mathcal{P}_2(M)$ induced from the projection, i.e.
\begin{equation} 
\label{1}
\uppp \circ \iota = \pi_{*}\text{.}
\end{equation}
It sends all 1-morphisms and 2-morphisms which are not in the image of $\iota$  to identities.   
In this section we show the following result.

\begin{proposition}
\label{prop2}
The 2-functor $\uppp$ is an equivalence of 2-categories. 
\end{proposition}

In order to prove this  we introduce an inverse 2-functor
\begin{equation*}
s \maps  \mathcal{P}_2(M) \to \upp{\pi}{2}{M}\text{.}
\end{equation*}
Since the 2-functor $\uppp$ is surjective on objects, we call $s$ the \emph{section 2-functor}. To define $s$, we  lift points, paths and bigons in $M$ along the surjective submersion $\pi$, and use the jumps and the several 2-morphisms of the codescent 2-groupoid whenever no \quot{global} lifts exist.

For preparation we need the following technical lemma.
\begin{lemma}
\label{lem3}
Let $\gamma \maps x \to y$ be a path in $M$, and let $\tilde x, \tilde y \in Y$ be lifts of the endpoints, i.e. $\pi(\tilde x)=x$ and $\pi(\tilde y)=y$.
\begin{enumerate}
\item[(a)] 
There exists a 1-morphism $\tilde \gamma \maps \tilde x \to \tilde y$ in $\upp{\pi}{2}{M}$ such that $\uppp(\tilde\gamma) = \gamma$.
\item[(b)]
Let $\tilde \gamma \maps \tilde x \to \tilde y$ and $\tilde \gamma' \maps \tilde x \to \tilde y$ be two such 1-morphisms. Then, there exists a unique 2-isomorphism $A \maps \tilde\gamma \Rightarrow \tilde \gamma'$ in $\upp{\pi}{2}{M}$ such that $\uppp(A)=\id_{\gamma}$.
\end{enumerate}
\end{lemma}

The assertion (a) is proven as \cite[Lemma 2.15]{schreiber3}. The proof of (b) requires some preparation.

\begin{lemma}
\label{app2pp_lem1}
Let $p\in M$ be a point and $a,b\in Y$ with $\pi(a)=\pi(b)=p$. Let $\alpha \maps a \to b$ and $\beta \maps a \to b$ be 1-morphisms in $\upp{\pi}{2}{M}$ which are  compositions of jumps.
\begin{enumerate}
\item[(a)]
There exists a 2-isomorphism $\Xi \maps \alpha\Rightarrow \beta$ with  $\uppp(\Xi)=\id_{\id_{p}}$.

\item[(b)] Any 2-isomorphism $\Xi \maps \alpha \Rightarrow \beta$ with $\uppp(\Xi)=\id_{\id_{p}}$ can be represented by a composition of 2-morphisms of type (1c).

\item[(c)]
The 2-isomorphism from (a) is unique.

\end{enumerate}
\end{lemma}

\begin{proof}
It is easy to construct the 2-isomorphism of (a) using only 2-isomorphisms of type (1c) and their inverses. To show (b) let  $\Xi \maps \alpha \Rightarrow \beta$ be a 2-isomorphism with $\uppp(\Xi)=\id_{\id_{p}}$, represented by a composition of 2-morphisms of any type. In the following we draw pasting diagrams to demonstrate that all 2-morphisms of types (1a), (1b) and (1d) can subsequently  be killed. 

To prepare some machinery notice that identification (III) imposes axiom (T2) for the \pt\ $\Gamma$,
which is, for any bigon $\Sigma \maps  \Theta_1 \Rightarrow \Theta_2$ in $Y^{[2]}$, the identity:
\begin{equation}
\label{app2pp_3}
\alxydim{@C=1.2cm@R=1.2cm}{\pi_1(\alpha) \ar[r]^{\pi_1(\Theta_1)} \ar[d]_{\alpha} & \pi_1(\alpha') \ar[d]^{\alpha'}
\ar@{=>}[dl]|{\Theta_1} \\ \pi_2(\alpha) \ar@/_3pc/[r]_{\pi_2(\Theta_2)}="2"   \ar[r]^{\pi_2(\Theta_1)}="1" & \pi_2(\alpha') \ar@{=>}"1";"2"|{\pi_2(\Sigma)}}
=
\alxydim{@C=1.2cm@R=1.2cm}{\pi_1(\alpha) \ar@/^3pc/[r]^{\pi_1(\Theta_1)}="1" \ar[r]_{\pi_1(\Theta_2)}="2"
\ar@{=>}"1";"2"|{\pi_1(\Sigma)} \ar[d]_{\alpha} & \pi_1(\alpha') \ar[d]^{\alpha'}
\ar@{=>}[dl]|{\Theta_2} \\ \pi_2(\alpha) \ar[r]_{\pi_2(\Theta_2)} & \pi_2(\alpha')}
\end{equation}
In the same way, identification (IV) imposes the axiom for the modification $\id_\iota \Rightarrow \Delta^{*}\Gamma$, which is, for any  path $\gamma \maps a \to b$ in $Y$, the identity
\begin{equation}
\label{app2pp_1}
\alxydim{@C=1.5cm@R=1.5cm}{a \ar[dr]|{\gamma}="5" \ar@/_4pc/[d]_{\Delta(a)}="1" \ar[r]^{\gamma} \ar[d]|{\id_a^{*}}="2" \ar@{<=}"1";"2"|-{\Delta_a} & b \ar[d]^{\id_b^{*}}
\ar@{=>}"5"|{r_{\gamma}^{*}} \\ a \ar[r]_{\gamma} & b \ar@{=>}"5";[l]|-{l_\gamma^{*-1}}}
=
\alxydim{@C=1.5cm@R=1.5cm}{a \ar[r]^{\gamma} \ar[d]_{\Delta(a)} & b \ar@/^4pc/[d]^{\id_b^{*}}="3" \ar[d]|{\Delta(b)}="4" \ar@{=>}"3";"4"|-{\Delta_b}
\ar@{=>}[dl]|{\Delta(\gamma)} \\ a \ar[r]_{\gamma} & b}
\end{equation}
Using \erf{app2pp_1} we can write the identity 2-morphism associated to the path $\gamma$ in a very fancy way, namely
\begin{equation}
\label{app2pp_2}
\id_{\gamma}^{\circledast} = \alxydim{@R=0.1cm@C=2cm}{&b \ar@/^1pc/[dr]|{\id_b^{*}}="1" \ar@{=>}[dd]|{\Delta(\gamma)} \ar[dr]_{\Delta(b)}="2"&\\a \ar@/^3.5pc/[rr]^{\gamma}="5" \ar@/_3.5pc/[rr]_{\gamma}="6" \ar@/_1pc/[dr]|{\id_a^{*}}="4" \ar[dr]^>>>>>>>>>{\Delta(a)}="3" \ar[ur]|{\gamma}&&b\text{.}\\&a \ar[ur]|{\gamma}& \ar@{=>}"1";"2" \ar@{=>}"3";"4" \ar@{=>}"5";[uul] \ar@{=>}[l];"6"}
\end{equation}

Now suppose that $\Sigma \maps \gamma \Rightarrow \gamma'$ is some 2-morphism of type (1a) that we want to kill. We write $\Sigma = \Sigma \circledast \id_{\gamma}^{\circledast}$ and use \erf{app2pp_2}. Using the naturality of the 2-morphism $l_{\gamma}^{*}$ claimed by identification (I) we have
\begin{equation*}
\Sigma = \alxydim{@R=0.1cm@C=2cm}{&b \ar@/^1pc/[dr]|{}="1" \ar@{=>}[dd]|{\Delta(\gamma)} \ar[dr]_{}="2"&\\a \ar@/^3.5pc/[rr]^{\gamma}="5" \ar@/_3.5pc/[rr]_{\gamma'}="6" \ar@/_1pc/[dr]|{}="4" \ar[dr]^{}="3" \ar[ur]|{\gamma}&&b\\&a \ar@/_1pc/[ur]_<<<<<{\gamma'}="8" \ar@/^0.9pc/[ur]|{\gamma}="7" \ar@{=>}"7";"8"|-{\Sigma}& \ar@{=>}"1";"2" \ar@{=>}"3";"4" \ar@{=>}"5";[uul]|{} \ar@{=>}[l];"6"} 
= \alxydim{@R=0.1cm@C=2cm}{&b \ar@/^1pc/[dr]|{\id_b^{*}}="1" \ar@{=>}[dd]|{\Theta_2} \ar[dr]_{\Delta(b)}="2"&\\a \ar@/^3.5pc/[rr]^{\gamma}="5" \ar@/_3.5pc/[rr]_{\gamma'}="6" \ar@/_1pc/[dr]|{\id_a^{*}}="4" \ar[dr]^>>>>>>>>>{\Delta(a)}="3" \ar[ur]|{\gamma}&&b\\&a \ar[ur]|{\gamma'}& \ar@{=>}"1";"2" \ar@{=>}"3";"4" \ar@{=>}"5";[uul] \ar@{=>}[l];"6"}
\end{equation*}
where the second identity is obtained from  \erf{app2pp_3} by taking $\Theta_1 \df \Delta(\gamma)$ and $\Theta_2  \df  (\gamma,\gamma')$ which is only possible because we have assumed that $\uppp(\Sigma)=\id_{\id_p}$. We can thus kill every 2-morphism of type (1a).

 Suppose now that  $\Psi \maps \mu \Rightarrow \nu$ is a 2-morphism of type (1b). To kill it we need identification (IV), namely the axiom for the modification $\pi_{23}^{*}\Gamma\circ \pi_{12}^{*}\Gamma \Rightarrow \pi_{13}^{*}\Gamma$. For any path $\Theta \maps \Xi \to \Xi'$ in $Y^{[3]}$, the corresponding pasting diagram is
\begin{equation}
\label{app2pp_5}
\hspace{-0.52cm}
\alxydim{@C=1.3cm}{\pi_1(\Xi) \ar[r]^-{\pi_1(\Theta)} \ar[d]|{\pi_{12}(\Xi)} \ar@/_3.4pc/[dd]_{\pi_{13}(\Xi)}="1" & \pi_1(\Xi') \ar@{=>}[dl]|{\pi_{12}(\Theta)} \ar[d]^{\pi_{12}(\Xi')} \\ \pi_2(\Xi) \ar[r]|-{\pi_2(\Theta)} \ar[d]|{\pi_{23}(\Xi)} & \pi_2(\Xi') \ar@{=>}[dl]|{\pi_{23}(\Theta)} \ar[d]^{\pi_{23}(\Xi')} \\ \pi_3(\Xi) \ar[r]_-{\pi_3(\Theta)} & \pi_3(\Xi') \ar@{=>}[ul];"1"|{\Xi}}
=
\alxydim{}{\pi_1(\Xi) \ar[r]^{\pi_{1}(\Theta)} \ar[dd]_{\pi_{13}(\Xi)} & \pi_1(\Xi') \ar@{=>}[ddl]|{\pi_{13}(\Theta)} \ar[dd]|{\pi_{13}(\Xi')}="1" \ar@/^1.3pc/[rd]^{\pi_{12}(\Xi')} \\ && \pi_2(\Xi') \ar@{=>}"1"|{\Xi'} \ar@/^1.3pc/[dl]^{\pi_{23}(\Xi')} \\ \pi_3(\Xi) \ar[r]_{\pi_3(\Theta)} & \pi_3(\Xi')}
\end{equation}
Here we suppress writing the associators and the bracketing of the 1-morphisms. Using this identity we have
\begin{equation}
\label{app2pp_4}
\Psi = \alxydim{@C=1.3cm}{\pi_1(\mu) \ar[r]^-{\pi_1(\Psi)} \ar[d]|{} \ar@/_3.4pc/[dd]_{\mu}="1" & \pi_1(\nu) \ar@/^3.4pc/[dd]^{\nu}="2" \ar@{=>}"2";[d]|{} \ar@{=>}[dl]|{\pi_{12}(\Theta)} \ar[d]|{} \\ c \ar[r]|-{\id_c} \ar[d]|{} & c \ar@{=>}[dl]|{\pi_{23}(\Theta)} \ar[d]|{} \\ \pi_2(\mu) \ar[r]_-{\pi_2(\Psi)} & \pi_2(\nu) \ar@{=>}[ul];"1"|{}}
\end{equation}
for $c\in Y$ an arbitrary  point with $\pi(c)=p$ and $\Theta  \df (\pi_1(\Psi),\id_c,\pi_2(\Psi))$ which is only possible because $\uppp(\Psi)=\id_{\id_p}$. 

We can now assume that the 2-morphism $\Xi \maps \alpha \Rightarrow \beta$ we started with contains no 2-morphism of type (1a) and by \erf{app2pp_4} only those 2-morphism $\Theta=(\gamma_1,\gamma_2)$ for which either  $\gamma_1$ or $\gamma_2$ is the identity path of the point $c$. If both $\gamma_1$ and $\gamma_2$ are identity paths, we can replace $\Theta$ according to \erf{app2pp_1} by two 2-morphisms of type (1d). It is now a combinatorial task to kill all 2-morphisms which start or end on paths, in particular all 2-morphisms of type (2b). Then one kills all 2-morphisms of types (1d) and the remaining unifiers $l^*_{\beta}$ and $r^{*}_{\beta}$. Finally, all associators $a^{*}$ can be killed using their naturality with respect to 2-morphisms of type (1c). This proves (b).

To prove (c) we assume that $\Xi' \maps \alpha \Rightarrow \beta$ is any 2-isomorphism with $\uppp(\Xi)=\id_{\id_{p}}$. By (b) we can assume that $\Xi'$ is composed only of 2-isomorphisms of type (1c). It is straightforward to see that two such compositions can be transformed into each other if six identities are satisfied: two \emph{bubble} identities and four \emph{fusion} identities.  The two bubble identities are
\begin{equation*}
\alxydim{@C=-0.3cm@R=0.6cm}{& \pi_2(\Psi) \ar[dr] & \\ \pi_1(\Psi) \ar[dr] \ar[ur] \ar[rr]|{}="1" \ar@{=>}[ur];"1"|-{\Psi} \ar@{=>}"1";[dr]|-{\bar\Psi} && \pi_3(\Psi) \\ & \pi_2(\Psi) \ar[ur]&} = \id_{\pi_{23}(\Psi)\circ\pi_{12}(\Psi)}
\;\;\text{ and }\;\;
\alxydim{@C=0.25cm}{\pi_1(\Psi) \ar@/^3pc/[rr]^{\pi_{13}(\Psi)}="1" \ar@/_3pc/[rr]_{\pi_{13}(\Psi)}="2" \ar[r] & \pi_2(\Psi) \ar[r] & \pi_{3}(\Xi) \ar@{=>}"1";[l]|{\bar\Psi} \ar@{=>}[l];"2"|{\Psi}}= \id_{\pi_{13}(\Psi)}\text{.}
\end{equation*}
They follow from the fact that the 2-morphisms of type (1c) are invertible. The first fusion identity is identification (V1),
\begin{equation*}
\alxydim{@C=1.5cm@R=1.7cm}{\pi_2(\Psi) \ar[r]^{\pi_{23}(\Psi)} \ar@{<-}[d]_{\pi_{12}(\Psi)}
& \pi_3(\Psi) \ar[d]^{\pi_{34}(\Psi)} \\ \pi_1(\Psi)  \ar[ur]|{\pi_{13}(\Psi)}="2" \ar@{=>}[u];"2"|{\pi_{123}(\Psi)} \ar[r]_{\pi_{14}(\Psi)}="1"
&  \ar@{<=}"1";"2"|{\pi_{134}(\Psi)} \pi_4(\Psi)}
=
\alxydim{@C=1.5cm@R=1.7cm}{\pi_2(\Psi) \ar[dr]|{\pi_{24}(\Psi)}="2" \ar[r]^{\pi_{23}(\Psi)} \ar@{<-}[d]_{\pi_{12}(\Psi)}
& \pi_3(\Psi) \ar[d]^{\pi_{34}(\Psi)}  \ar@{<=}"2";[]|{\pi_{234}(\Psi)}
 \\  \pi_1(\Psi)  \ar[r]_{\pi_{14}(\Psi)}="1"
\ar@{<=}"1";"2"|{\pi_{124}(\Psi)}& \pi_4(\Psi)\text{.}}
\end{equation*}
The other three fusion identities are analogous identities for formal inverses $\bar\Psi$ and mixtures of $\Psi$ and $\bar\Psi$.
\end{proof}

\begin{proof}[Proof of Lemma \ref{lem3} (b)]
Now let $\gamma \maps x \to y$ be a path in $M$, and let $\tilde x, \tilde y \in Y$ be lifts of the endpoints, i.e. $\pi(\tilde x)=x$ and $\pi(\tilde y)=y$.
We compare the two lifts $\tilde\gamma$ and $\tilde\gamma'$ of $\gamma$ in the following way.
Let $P  \subset M$ be the set of  points over whose fibre either $\tilde \gamma_1$ or $\tilde\gamma_2$ has a jump. The set $P$ is finite and ordered by the orientation of the path $\gamma$, so that we may put $P= \lbrace p_0,...,p_n \rbrace$ with $p_0=x$ and $p_n=y$. Now we write 
\begin{equation*}
\gamma= \gamma_n \circ ... \circ \gamma_1
\end{equation*}  
for paths $\gamma_k \maps  p_{k-1} \to p_k$. We also write $\tilde\gamma$  as a composition of lifts $\tilde \gamma_k \maps a_k \to b_k$ of $\gamma_k$ and (possibly multiple) jumps $b_k \to \alpha_{k+1}$ located over the points $p_k$;  analogously for $\tilde\gamma'$. This defines jumps $\alpha_k  \df  (a_k,a_k')$ and $\beta_k \df (b_k,b_{k}')$.
Now, over the paths $\gamma_k$, we take 2-isomorphisms
\begin{equation}
\label{app2pp_6}
\quadrat{a_k}{b_k}{a_k'}{b_k'}{\tilde\gamma_k}{\alpha_{k}}{\beta_k}{\tilde \gamma_k'}{\Theta}
\end{equation}
with $\Theta  \df  (\tilde\gamma_k,\tilde\gamma_k')$. Over the points $p_k$ we need 2-isomorphisms of the form
\begin{equation}
\label{app2pp_7}
\alxydim{@R=1.2cm@C=0.2cm}{& b_k=a_{k+1} \ar[dr]^{\alpha_{k+1}} & \\ b_k' \ar@{<-}[ru]^{\beta_k} \ar[rr]_-{}="1" && a'_{k+1} \ar@{=>}[ul];"1"|{}}
\quad\text{, }\quad
\alxydim{@R=1.2cm@C=0.2cm}{b_k \ar[dr]_{\beta_k}="1" \ar[rr] && a_{k+1} \ar[dl]^{\alpha_{k+1}}\\ & b_k' = a'_{k+1} & \ar@{=>}[u];"1"|{}}
\quad\text{ or }\quad
\quadrat{b_k}{a_{k+1}}{b_k'}{a_{k+1}'}{}{\beta_k}{\alpha_{k+1}}{}{}
\end{equation}   
the first whenever $\tilde\gamma'$ has  jumps over $p_{k}$  and $\tilde\gamma$ has not, the second whenever $\tilde\gamma$ has  jumps and $\tilde\gamma'$ has not, and the third whenever both lifts have jumps. By Lemma \ref{app2pp_lem1} these 2-isomorphisms exist and are unique.  Then, all of the four diagrams above can be put next to each other; this defines a 2-isomorphism $\tilde\gamma \Rightarrow \tilde\gamma'$. We call the 2-morphism constructed like this the canonical 2-morphism. 

It remains to show that every 2-morphism $A \maps \tilde\gamma \Rightarrow \tilde \gamma'$ with $\uppp(A)=\id_{\gamma}$ is equal to this canonical 2-morphism. First, we kill all bigons contained in $A$  by the argument given in the proof of Lemma \ref{app2pp_lem1}. We consider two cases:
\begin{enumerate}
\item 
$A$ contains no paths except those contained in $\tilde\gamma$ or $\tilde\gamma'$. In this case $A$ is already equal to the canonical 2-morphism. Namely, each of the pieces $\tilde\gamma_k$ or $\tilde\gamma'_k$ can only be target or source of a 2-morphism of type (1b). These are now necessarily  the pieces \erf{app2pp_6}. It remains to consider the 2-morphisms between the jumps. But these are by Lemma \ref{app2pp_lem1} equal to the pieces  \erf{app2pp_7}. This shows that $A$ is  the canonical 2-morphism. 

\item
There exists a path $\gamma_0$ in $\upp{\pi}{2}{M}$ which is target or source of some 2-morphism contained in $A$ but not contained in $\tilde\gamma$ or $\tilde\gamma'$. In this case there exists a 1-morphism $\tilde\gamma_o \maps \tilde x \to \tilde y$ together with 2-morphisms $A_1 \maps \tilde\gamma \Rightarrow \tilde\gamma_0$ and $A_2 \maps \tilde\gamma_0 \Rightarrow \tilde\gamma'$ such that $A=A_2 \bullet A_1$. By iteration, we can decompose $A$ in a vertical composition of 2-morphisms which fall into case 1, i.e. into a vertical composition of canonical 2-morphisms.
\end{enumerate}

It remains to conclude with the observation that the vertical composition $A_2 \bullet A_1$ of two canonical 2-morphisms is again canonical. 
\end{proof}

 To construct the 2-functor $s$
 we fix   an open cover $\lbrace U_i \rbrace_{i\in I}$ of $M$ together with smooth sections $\sigma_i \maps U_i \to Y$, we fix choices of lifts $s(p)\in Y$ for all points $p\in M$, and we  fix for every path $\gamma \maps x \to y$ in $M$ a 1-morphism $s(\gamma) \maps s(x) \to s(y)$ in $\upp{\pi}{2}{M}$. Such lifts exist according to Lemma \ref{lem3} (a). For the identity 1-morphisms $\id_x$ we  choose the identity 1-morphisms $\id_{s(x)}^{*}$. Moreover, we require $s(\gamma^{-1}) = s(\gamma)^{-1}$, meaning that $s(\gamma)^{-1}$ is the reverse order composition of the inverses $\tilde\gamma^{-1}$ of the involved paths $\tilde\gamma$, and of the inverses $\bar\alpha$ of all involved jumps $\alpha$.   These choices define $s$ on objects and 1-morphisms.

Now let $\Sigma \maps  \gamma_1 \Rightarrow \gamma_2$ be a bigon in $M$. Its image $\Sigma([0,1]^2) \subset M$ is compact and hence  covered by open  sets indexed by a \emph{finite} subset $J \subset I$. We choose a decomposition of $\Sigma$ in a vertical and horizontal composition of bigons $\lbrace \Sigma_j \rbrace_{j \in J}$ such that $\Sigma_j([0,1]^2) \subset U_j$. Then we define $s(\Sigma)$ to be composed from  2-morphisms $s(\Sigma_{j})$ in the same way as $\Sigma$ was composed from the $\Sigma_j$. It thus remains to define the 2-functor $s$ on bigons $\Sigma$ which are contained in an open set $U$ which has a section $\sigma \maps U \to Y$. 
We define for such a bigon 
\begin{equation*}
s
\quad : \quad
\bigon{x}{y}{\gamma_1}{\gamma_2}{\Sigma}
\quad \mapsto \quad
\alxy{s(x) \ar@/_4.5pc/[rrr]_{s(\gamma_2)}="4" \ar@/^4.5pc/[rrr]^{s(\gamma_1)}="0" \ar[r] & \sigma(x) \ar@/^1.5pc/[r]^{\sigma(\gamma_1)}="1"
\ar@/_1.5pc/[r]_{\sigma(\gamma_2)}="2" \ar@{=>}"1";"2"|{\sigma(\Sigma)} & \sigma(y) \ar[r] & s(y) \ar@{=>}"0";"1"|{} \ar@{=>}"2";"4"|{}}
\end{equation*}
where the unlabelled 1-morphisms are the obvious jumps, and the unlabelled 2-morphisms are the unique 2-isomorphisms from Lemma \ref{lem3} (b).

The 2-functor $s \maps  \mathcal{P}_2(M) \to \upp{\pi}{2}{M}$ whose structure we have defined above is not  strict. While its unitor is trivial because we have by definition $s(\id_x)=\id^{*}_{s(x)}$,  its compositor $c_{\gamma_1,\gamma_2} \maps  s(\gamma_2) \circ s(\gamma_1) \Rightarrow s(\gamma_2 \circ \gamma_1)$ is defined to be the unique 2-isomorphism from Lemma \ref{lem3} (b).  
All axioms for the 2-functor $s$ follow from the uniqueness of these 2-isomorphisms.

For later purpose, we note the following consequence of the definitions.

\begin{lemma}
\label{comps}
If $\gamma:x \to y$ is a path in $M$, then  the compositor $c_{\gamma,\gamma^{-1}}$ of $s$ is a composition of 2-morphisms of types (2a) and (2b), type (1d), and those 2-morphisms $\Xi \in Y^{[3]}$ of type (1c) that are in the image of $\Delta_{121}: Y^{[2]} \to Y^{[3]}:(a,b) \mapsto (a,b,a)$.
\end{lemma}

Now we can proceed with the remaining proof of the main result of this  section.

\begin{proof}[Proof of Proposition \ref{prop2}]
By construction we have $\uppp \circ s = \id_{\mathcal{P}_2(M)}$. It remains to construct a \pe
\begin{equation*}
\zeta \maps  s \circ \uppp \to \id_{\upp{\pi}{2}{M}}\text{.}
\end{equation*}
We define $\zeta$ on both basic   1-morphisms. Its component at a path  is
\begin{equation*}
\zeta \quad \maps \quad \alxydim{}{a \ar[r]^{\gamma} & b}
\quad\mapsto\quad
\quadrat{s(\pi(a))}{s(\pi(b))}{a}{b}{s(\pi_{*}\gamma)}{}{}{\gamma}{}
\end{equation*}
where the unlabelled 1-morphisms are again the obvious jumps, and the 2-isomorphism is the unique one. If $s(\pi_{*}\gamma)$ happens to be just a path, this 2-isomorphism is just of type (1b).
The component of $\zeta$ at a jump is
\begin{equation*}
\zeta \quad \maps \quad \alxydim{}{\pi_1(\alpha) \ar[r]^{\alpha} & \pi_2(\alpha)}
\quad\mapsto\quad
\dreieck{\pi_1(\alpha)}{s(x)}{\pi_2(\alpha)}{}{}{}{}
\end{equation*}
with $x \df \pi(\pi_1(\alpha))=\pi(\pi_2(\alpha))$; this is just one 2-isomorphism of type (1c). For some general 1-morphism, $\zeta$ puts the 2-isomorphisms above next to each other; this way axiom (T1) is automatically satisfied. Axiom (T2) follows again from the uniqueness of the 2-morphisms we have used.

In order to show that $\zeta$ is invertible we need to find another \pt\ $\xi  \maps  \id_{\upp{\pi}{2}{M}} \to s \circ \uppp$ together with invertible modifications $i_{\zeta} \maps  \xi \circ \zeta \Rightarrow \id_{s \circ \uppp }$ and $j_{\zeta} \maps  \id_{\id_{\upp{\pi}{2}{M}}}\Rightarrow\zeta\circ \xi$ that satisfy the \zickzack identities. The \pt\ $\xi$ can be defined in the same way as $\zeta$ just by turning the diagrams upside down, using the formal inverses.  The modifications $i_{\zeta}$ and $j_\zeta$ assign to a point $a\in Y$ the  2-isomorphisms
\begin{equation*}
\alxydim{@R=0.3cm}{&a \ar@/^/[dr]^{\xi(a)} &\\s(\pi(a)) \ar[rr]|{\Delta(s(\pi(a)))}="1" \ar@/_1.9pc/[rr]_{\id^{*}_{s(\pi(a))}}="2" \ar@{=>}[ur];"1"|{} \ar@{=>}"1";"2" \ar@/^/[ru]^{\zeta(a)}&&s(\pi(a))}
\quad\text{ and }\quad
\alxydim{@R=0.3cm}{a \ar[rr]|{\Delta(a)}="1" \ar@/^1.9pc/[rr]^{\id^{*}_a}="2" \ar@{=>}"1";[dr]|{} \ar@{=>}"2";"1" \ar@/_/[rd]_-{\zeta(a)}&&a\\&s(\pi(a)) \ar@/_/[ur]_-{\xi(a)} &}
\end{equation*}
that combine 2-isomorphisms of type (1c) and (1d). The zigzag identities are satisfied due to the uniqueness of 2-isomorphisms we have used.
\end{proof}

\begin{corollary}
\label{co1}
The section 2-functor $s \maps \mathcal{P}_2(M) \to \upp{\pi}{2}{M}$ is independent
(up to \pe) of all choices, namely the choice of lifts of points and 1-morphisms, the choice of the open cover, and the choice of local sections.
\end{corollary}

This follows from the fact that two weak inverses of a 1-morphism in a 2-category are automatically 2-isomorphic.

\subsection{Pairing with Descent Data}

\label{sec2_3}

In this section we relate the \codescent\ $\upp{\pi}{2}{M}$ to the descent 2-category $\trans{i}{\mathrm{Gr}}{2}$ defined in Section \ref{sec:des} in terms of a strict 2-functor 
\begin{equation*}
R \maps  \trans{i}{\mathrm{Gr}}{2} \to \mathrm{Funct}(\upp{\pi}{2}{M},T)\text{.}
\end{equation*}
This 2-functor expresses that the 2-groupoid $\upp{\pi}{2}{M}$  is \quot{$T$-dual} to the  descent 2-category; this justifies the term \emph{co}descent 2-groupoid.

The 2-functor $R$ labels the structure of the codescent 2-groupoid by  descent data in a natural way. To start with, let $(\mathrm{triv},g,\psi,f)$ be a descent object. Its image under $R$ is a 2-functor 
\begin{equation*}
R_{(\mathrm{triv},g,\psi,f)} \maps  \upp{\pi}{2}{M} \to T
\end{equation*}
 defined as follows. 
To an object
$a\in Y$ it assigns the object $\mathrm{triv}_i(a)$ in $T$.
On basic 1-morphisms it is defined by the following assignments:
\begin{eqnarray*}
\alxydim{}{a \ar[r]^{\gamma} & a'}
&\quad\mapsto\quad&
\alxydim{@C=1.5cm}{\mathrm{triv}_i(a) \ar[r]^{\mathrm{triv}_i(\gamma)} & \mathrm{triv}_i(a')}
\\[-1cm]\hspace{4cm}&&\hspace{8cm}
\end{eqnarray*}
\begin{eqnarray*}
\alxydim{}{\pi_1(\alpha) \ar[r]^{\alpha} & \pi_2(\alpha)}
&\quad\mapsto\quad&
\alxydim{@C=1.5cm}{\pi_1^{*}\mathrm{triv}_i(\alpha) \ar[r]^{g(\alpha)} & \pi_2^{*}\mathrm{triv}_i(\alpha)\text{.}}
\\[-0.3cm]\hspace{4cm}&&\hspace{8cm}
\end{eqnarray*}
To a formal composition of basic 1-morphisms it assigns the composition of the respective images, and to the formal identity $\id_a^{*}$ at a point $a\in Y$ it assigns $\id_{\mathrm{triv}_i(a)}$. 
On the basic  2-morphisms of essential types (1a) to (1d)  it is defined by the following assignments:
\begin{eqnarray*}
\bigon{a}{b}{\gamma_1}{\gamma_2}{\Sigma}
&\quad \mapsto \quad&
\bigon{\mathrm{triv}_i(a)}{\mathrm{triv}_i(b)}{\mathrm{triv}_i(\gamma_1)}{\mathrm{triv}_i(\gamma_2)}{\mathrm{triv}_i(\Sigma)}
\\[-0.8cm]\hspace{6cm}&&\hspace{7cm}
\end{eqnarray*}

\begin{eqnarray*}
\alxydim{@C=0.8cm@R=1.2cm}{\pi_1(\alpha) \ar[d]_{\alpha} \ar[r]^{\pi_1(\Theta)} & \pi_1(\alpha') \ar[d]^{\alpha'}
\ar@{=>}[dl]|{\Theta} \\ \pi_2(\alpha) \ar[r]_{\pi_2(\Theta)} & \pi_2(\alpha')}
&\quad \mapsto \quad&
\alxydim{@C=1.2cm@R=1.2cm}{\pi_1^{*}\mathrm{triv}_i(\alpha) \ar[d]_{g(\alpha)} \ar[r]^{\pi_1^{*}\mathrm{triv}_i(\Theta)} & \pi_1^{*}\mathrm{triv}_i(\alpha') \ar[d]^{g(\alpha')}
\ar@{=>}[dl]|{g(\Theta)} \\ \pi_2^{*}\mathrm{triv}_i(\alpha) \ar[r]_{\pi_2^{*}\mathrm{triv}_i(\Theta)} & \pi_2^{*}\mathrm{triv}_i(\alpha')}
\\[-0.8cm]\hspace{6cm}&&\hspace{7cm}
\end{eqnarray*}

\begin{eqnarray*}
\dreieck{\pi_1(\Xi)}{\pi_2(\Xi)}{\pi_{3}(\Xi)}{\pi_{12}(\Xi)}{\pi_{23}^{*}(\Xi)}{\pi_{13}(\Xi)}{\Xi}
&\quad \mapsto \quad&
\alxydim{@C=-0.5cm@R=1.2cm}{& \pi_2^{*}\mathrm{triv}_i(\Xi) \ar[dr]^{\pi_{23}^{*}g^{*}(\Xi)}& \\ \pi_1^{*}\mathrm{triv}_i(\Xi) \ar[ur]^{\pi_{12}^{*}g(\Xi)} \ar[rr]_{\pi_{13}^{*}g(\Xi)}="1" && \pi_{3}^{*}\mathrm{triv}_i(\Xi) \ar@{=>}[ul];"1"|{f(\Xi)}}
\\[-0.8cm]\hspace{6cm}&&\hspace{7cm}
\end{eqnarray*}

\begin{eqnarray*}
\alxydim{}{\id_a^{*} \ar@{=>}[r]^-{\Delta_a} & \Delta(a)}
&\quad \mapsto \quad&
\alxydim{}{\id_{\mathrm{triv}_i(a)} \ar@{=>}[r]^-{\psi(a)} & \Delta^{*}g(a)\text{.}}
\\[-0.3cm]\hspace{6cm}&&\hspace{7cm}
\end{eqnarray*}
To the basic 2-morphisms of technical type (2a) it assigns associators and \unifier s of the 2-category $T$. To those of type (2b) it assigns unitors and compositors of the 2-functor $i$, i.e.
\begin{eqnarray*}
\alxydim{}{\id_a \ar@{=>}[r]^{u_a^{*}} & \id_a^{*}}
&\;\mapsto\;&
\alxydim{@C=1.5cm}{\mathrm{triv}_i(\id_{a}) \ar@{=>}[r]^{u^i_{\mathrm{triv}(a)}} & \id_{\mathrm{triv}_i(a)}}
\\[-1cm]\hspace{4cm}&&\hspace{9cm}
\end{eqnarray*}
\begin{eqnarray*}
\alxydim{@C=1cm}{\gamma_2 *\gamma_1 \ar@{=>}[r]^{c^{*}_{\gamma_1,\gamma_2}} & \gamma_2 \circ \gamma_1 }
&\;\mapsto\;&
\alxydim{@C=2.3cm}{\mathrm{triv}_i(\gamma_2) \circ \mathrm{triv}_i(\gamma_1) \ar@{=>}[r]^-{c^i_{\mathrm{triv}(\gamma_1),\mathrm{triv}(\gamma_2)}} & \mathrm{triv}_i(\gamma_2 \circ \gamma_1)\text{.} }
\\[-0.3cm]\hspace{4cm}&&\hspace{9cm}
\end{eqnarray*}
Finally, some  formal horizontal and vertical composition of 2-morphisms is assigned to the composition of the images of the respective basic 2-morphisms, the formal horizontal composition replaced by the horizontal composition $\circ$ of $T$, and the formal vertical composition replaced by the vertical composition $\bullet$ of $T$.

By construction, all these assignments are well-defined under the identifications we have declared under the 2-morphisms of $\upp{\pi}{2}{M}$:   
\begin{itemize}
\item 
They are well-defined under the identifications (I) due to the axioms of the 2-category $T$.
\item
They are well-defined under identifications (II) due to the axioms of the 2-functors $\mathrm{triv}$ and $i$.
\item
They are well-defined under identifications (III) due to the axioms of the \pt\ $g$.

\item
They are well-defined under identifications (IV) due to the axioms of the modifications $\psi$ and $f$.
\item
They are well-defined under the identifications (V) because these are explicitly assumed in the definition of descent objects, see diagrams \erf{2} and \erf{3}. 
\end{itemize}
We have now defined the 2-functor $R_{(\mathrm{triv},g,\psi,f)}$ on descent objects, 1-morphisms and 2-morphisms. Since for all points $a\in Y$ 
\begin{equation*}
R_{(\mathrm{triv},g,\psi ,f)}(\id_a^{*} ) = \id_{\mathrm{triv}_i(a)} = \id_{R_{(\mathrm{triv},g,\psi ,f)}(a)}\text{,} 
\end{equation*}
it has a trivial unitor. Furthermore,
\begin{equation*}
R_{(\mathrm{triv},g,\psi ,f)}(\gamma) \circ R_{(\mathrm{triv},g,\psi ,f)}(\beta) = R_{(\mathrm{triv},g,\psi ,f)}(\gamma * \beta)
\end{equation*}
for all composable 1-morphisms $\beta$ and $\gamma$ of any type, so that it also has a trivial compositor. Hence, the 2-functor $R_{(\mathrm{triv},g,\psi ,f)}$ is strict, and it is straightforward to see that the remaining axioms (F1) and (F2) are satisfied.

So far we have introduced a 2-functor associated to each descent object. Let us now introduce a \pt\
\begin{equation*}
R_{(h,\varepsilon)} \maps  R_{(\mathrm{triv},g,\psi,f)} \to R_{(\mathrm{triv}',g',\psi',f')}
\end{equation*} 
associated to a descent 1-morphism
\begin{equation*}
(h,\varepsilon) \maps (\mathrm{triv},g,\psi ,f) \to (\mathrm{triv}',g',\psi ',f')\text{.}
\end{equation*}
Its definition is as natural as the one of the 2-functor given before.
Its component at an object $a\in Y$ is the 1-morphism 
\begin{equation*}
h(a) \maps \mathrm{triv}_i(a) \to \mathrm{triv}_i'(a)\text{.}
\end{equation*}
Its components at basic 1-morphisms are given by the following assignments:
\begin{eqnarray*}
\alxy{a \ar[r]^-{\gamma} & a'}
&\quad \mapsto \quad&
\alxydim{@C=1.2cm@R=1.2cm}{\mathrm{triv}_i(a) \ar[d]_{h(a)} \ar[r]^{\mathrm{triv}_i(\gamma)}
& \mathrm{triv}_i(a') \ar@{=>}[dl]|{h(\gamma)} \ar[d]^{h(a')}
\\ \mathrm{triv}_i'(a) \ar[r]_{\mathrm{triv}_i'(\gamma)}
& \mathrm{triv}_i'(a')}
\\[-0.8cm]\hspace{5cm}&&\hspace{7cm}
\end{eqnarray*}

\begin{eqnarray*}
\alxydim{@R=1cm@C=0.8cm}{\pi_1(\alpha) \ar[r]^-{\alpha}
& \pi_2(\alpha)}
&\quad \mapsto \quad&
\alxydim{@C=1.2cm@R=1.2cm}{\pi_1^{*}\mathrm{triv}_i(\alpha)
\ar[d]_{\pi_1^{*}h(\alpha)}
\ar[r]^{g(\alpha)} & \pi_2^{*}\mathrm{triv}_i(\alpha)
\ar@{=>}[dl]|{\varepsilon(\alpha)}
\ar[d]^{\pi_2^{*}h(\alpha)}
\\ \pi_1^{*}\mathrm{triv}_i'(\alpha)
\ar[r]_{g'(\alpha)} & \pi_2^{*}\mathrm{triv}'_i(\alpha)}
\\[-0.3cm]\hspace{5cm}&&\hspace{7cm}
\end{eqnarray*}
For compositions of 1-morphisms, $R_{(h,\varepsilon)}$ puts the diagrams for the involved basic 1-morphisms next to each other. For example, to a composition $\gamma*\alpha$ of a jump $\alpha=(x,y)$ with a path $\gamma \maps y \to z$ it assigns the 2-isomorphism
\begin{equation*}
h(z) \circ (\mathrm{triv}_i(\gamma) \circ g(\alpha)) \Rightarrow (\mathrm{triv}_i'(\gamma)\circ g(\alpha)) \circ h(x)
\end{equation*}
which is (up to the obvious associators) obtained by first applying $h(\gamma)$ and then $\varepsilon(\alpha)$.  This way, axiom (T1) for the \pt\ $R_{(h,\varepsilon)}$, namely the compatibility with the composition of 1-morphisms, is automatically satisfied. It remains to prove the following. 

\begin{lemma}
The assignments $R_{(h,\varepsilon)}$ are compatible with the 2-morphisms of the codescent 2-groupoid in the sense of axiom (T2).
\end{lemma}

\begin{proof}
We check the compatibility separately for each basic 2-morphism. For the essential 2-morphisms it comes from the following properties of the descent 1-morphism $(h,\varepsilon)$:
\begin{itemize}
\item For type (1a) it comes from axiom (T2) for the \pt\ $h$.
\item For type (1b) it comes from the axiom for the modification   $\varepsilon$ and from axiom (T2) for the \pt\ $h$.
\item For types (1c) and (1d) it comes from the conditions \erf{5} and \erf{4} on the descent 1-morphism $(h,\varepsilon)$.

\end{itemize}
For the technical 2-morphisms it comes from properties of the 2-category $T$ and the one of the 2-functor $i$:
for type (2a) it is satisfied because the associators and \unifier s of $T$ are natural, and for type (2b) it is satisfied because the compositors and unitors of $i$ are natural.
\end{proof}

We have now described a 2-functor associated to each descent object and a \pt\ associated to each descent 1-morphism. Now let $(\mathrm{triv},g,\psi,f)$ and $(\mathrm{triv}',g',\psi',f')$ be descent objects and let $(h_1,\varepsilon_1)$ and $(h_2,\varepsilon_2)$ be two descent 1-morphisms between these. For a descent 2-morphism $E \maps (h_1,\varepsilon_1) \Rightarrow (h_2,\varepsilon_2)$
we  introduce now a  modification
\begin{equation*}
R_E \maps  R_{(h_1,\varepsilon_1)}\Rightarrow R_{(h_2,\varepsilon_2)}\text{.}
\end{equation*}
Its component at an object $a\in Y$ is the 2-morphism $E(a) \maps  h_1(a) \Rightarrow h_2(a)$. The axiom for $R_E$, the compatibility with 1-morphisms, is satisfied for paths because $E$ is a modification, and for jumps because of the diagram \erf{6} in the definition of descent 2-morphisms.

It is now straightforward to see the following statement.

\begin{proposition}
The assignments defined above furnish a strict 2-functor
\begin{equation*}
R \maps  \trans{i}{\mathrm{Gr}}{2} \to \mathrm{Funct}(\upp{\pi}{2}{M} , T)\text{.}
\end{equation*} 
\end{proposition}

The 2-functor $R$  \quot{represents} the descent 2-category in a 2-category of 2-functors; in fact in a faithful way. We recall from Section \ref{sec:des} that there is a 2-functor $V \maps  \trans{i}{\mathrm{Gr}}{2} \to \mathrm{Funct}(\mathcal{P}_2(Y),T)$ which is also a representation of the same kind (but not faithful). The relation between these two representations is the following observation.

\begin{lemma}
\label{lem1}
$\iota^{*} \circ R =V$, where $\iota^{*}$ denotes the composition with $\iota \maps  \mathcal{P}_2(Y) \to \upp{\pi}{2}{M}$.
\end{lemma}

From this point of view, the codescent 2-groupoid enlarges the path 2-groupoid $\mathcal{P}_2(Y)$ by additional 1-morphisms (the jumps) and additional 2-morphisms in such a way that it carries a faithful representation of the descent 2-category.

Now we put the two main aspects of the codescent 2-groupoid together, namely the representation 2-functor $R$, and its equivalence with the path 2-groupoid in terms of the  section 2-functor $s$:  the \emph{reconstruction 2-functor} 
\begin{equation*}
\mathrm{Rec}_{\pi}: \trans{i}{\mathrm{Gr}}{2} \to \loctrivfunct{i}{\mathrm{Gr}}{2}
\end{equation*}
is defined to be   the composition
\begin{equation*}
\alxydim{@C=1.1cm}{\trans{i}{\mathrm{Gr}}{2} \ar[r]^-{R} & \mathrm{Funct}(\upp{\pi}{2}{M},T)
\ar[r]^{s^{*}} & \mathrm{Funct}(\mathcal{P}_2(M),T)}\text{.}
\end{equation*}
Here, $s^{*}$
is the composition
with $s$. According to Corollary \ref{co1}, the reconstruction 2-functor is (up to pseudonatural equivalence) canonically attached to the surjective submersion $\pi \maps Y \to M$ and the 2-functor $i \maps \mathrm{Gr} \to T$.

In order to show that the reconstruction ends in the sub-2-category $\loctrivfunct{i}{\mathrm{Gr}}{2}$ of $\mathrm{Funct}(\mathcal{P}_2(M),T)$ it remains to equip, for each descent object $(\mathrm{triv},g,\psi,f)$, the reconstructed 2-functor 
\begin{equation*}
F  \df  R_{(\mathrm{triv},g,\psi,f)} \circ s
\end{equation*}
with a $\pi$-local $i$-trivialization $(\mathrm{triv},t)$. Clearly, we take the given 2-functor $\mathrm{triv}$ as the first ingredient and are left with the construction of a \pe
\begin{equation}
\label{37}
t \maps  \pi^{*}F \to \mathrm{triv}_i\text{.}
\end{equation}
This equivalence is simply defined by
\begin{equation*}
\alxydim{@C=0.7cm@R=0.7cm}{\mathcal{P}_2(Y)  \ar[rr]^{\pi_{*}} \ar[ddd]_{\mathrm{triv}} \ar[dr]_{\iota} && \mathcal{P}_2(M)   \ar[dd]^{s} \\ & \upp{\pi}{2}{M} \ar[rd]_{\id}="1" \ar[ur]^{\uppp} & \\&&\upp{\pi}{2}{M} \ar[d]^{R_{(\mathrm{triv},g,\psi,f)}} \\ \mathrm{Gr} \ar[rr]_{i} && T \ar@{=>}[uuu];"1"|-*+{\zeta}}
\end{equation*}
where $\zeta$ is the \pe\ from Section \ref{sec2_2}. The triangle on the top of the latter diagram is equation \erf{1}, and the remaining subdiagram expresses the equation
\begin{equation*}
\iota^{*}R_{(\mathrm{triv},g,\psi,f)}= \mathrm{triv}_i 
\end{equation*}
which follows from Lemma \ref{lem1}.

We conclude with a lemma about the reconstruction of 2-functors from normalized descent objects.

\begin{lemma}
\label{reconinversionstrict}
Suppose $T$ is a strict 2-category, $i:\mathrm{Gr} \to T$ is a strict 2-functor, and $(\mathrm{triv},g,\psi,f)$ is a normalized descent object.  
Then, the reconstructed 2-functor 
\begin{equation*}
R_{(\mathrm{triv},g,\psi,f)} \circ s: \mathcal{P}_2(M) \to T
\end{equation*}
is normalized in the sense explained in Appendix \ref{app1}.
\end{lemma}

\begin{proof}
We write $R :=  R_{(\mathrm{triv},g,\psi,f)}$ and $F := R \circ s$ for abbreviation.
Let $x\in M$. Since $R$ is strict and $s$ has trivial unitor, $F$ has a trivial unitor:
\begin{equation*}
u_x^{F} = u^{R}_{s(x)} \bullet R (u^{s}_x) = R (\id_{\id^{*}_{s(x)}}) =\id_{R(\id^{*}_{s(x)})} = \id_{\id_{\mathrm{triv}_i(s(x))}}\text{.}
\end{equation*}
Let $\gamma:x \to y$ be a path. Since $R$ is strict, we have
\begin{equation*}
c^{F}_{\gamma,\gamma^{-1}} = R(c^{s}_{\gamma,\gamma^{-1}}) \text{.}
\end{equation*}
By Lemma \ref{comps} the compositor $c^{s}_{\gamma,\gamma^{-1}}$ consists of 2-morphisms of types (2a), (2b), (1c) and (1d). Since $T$ and $i$ are strict, all 2-morphisms of technical type (2a) and (2b) are sent by $R$ to identities. The 2-morphisms of types (1c)  and (1d) are sent by $R$ to the components of $\psi$ and $f$, and these are identities since $(\mathrm{triv},g,\psi,f)$ is a normalized descent object.
\end{proof}

\setsecnumdepth{2}

\section{Local - Global Equivalence}

\label{sec2_4}
\label{sec:equivalence}

In this section we prove the main theorem of this article, namely that extraction and reconstruction establish an equivalence between locally defined descent data and globally defined 2-functors.

\subsection{Equivalence for a Fixed Cover}

Let $i \maps  \mathrm{Gr} \to T$ be a 2-functor from a strict 2-groupoid $\mathrm{Gr}$ to a 2-category $T$, and let $\pi \maps Y \to M$ be a surjective submersion.

\begin{proposition}
\label{prop:equiv}
The 2-functor
\begin{equation*}
\ex{\pi}  \maps  \loctrivfunct{i}{\mathrm{Gr}}{2} \to \trans{i}{\mathrm{Gr}}{2}
\end{equation*}
is an equivalence of 2-categories.
\end{proposition}

For the proof we shall choose a section 2-functor $s \maps  \mathcal{P}_2(M) \to \upp{\pi}{2}{M}$ defining the reconstruction 2-functor $\mathrm{Rec}_{\pi}$. We show that $\ex{\pi}$ and $\mathrm{Rec}_{\pi}$ form an equivalence of 2-categories.  
This is done in the following two lemmata.
\begin{lemma}
\label{sec2_lem1}
There is a \pe\ $\ex{\pi} \circ \mathrm{Rec}_{\pi} \cong \id_{\trans{i}{i}{2}}$.
\end{lemma}

\begin{proof}
Given a descent object $(\mathrm{triv},g,\psi,f)$ let us pass to the reconstructed 2-functor and extract its descent data $(\mathrm{triv}',g',\psi',f')$. We find immediately $\mathrm{triv}'=\mathrm{triv}$. Furthermore, the \pt\ $g'$ has the components
\begin{equation*}
g'\quad:\quad\alxydim{}{\alpha \ar[r]^-{\Theta} & \alpha'}
\quad\mapsto\quad
\alxydim{@C=-0.0cm@R=0.7cm}
{& \pi_1^{*}\mathrm{triv}_i(\alpha) \ar[dl]
 \ar@/^1.8pc/[dd]^{df}="1"
\ar[rr]^{\pi_1^{*}\mathrm{triv}_i(\Theta)} & \hspace{2cm}& \pi_1^{*}\mathrm{triv}_i(\alpha') \ar@{=>}[ddll]|{g(\Theta)}\ar[dr] \ar@/_1.8pc/[dd]|{}="2" \\ c_{\alpha} \ar[dr] &&&&c_{\alpha'} \ar[dl] \\ & \pi_2^{*}\mathrm{triv}_i(\alpha) \ar[rr]_{\pi_1^{*}\mathrm{triv}_i(\Theta)} && \pi_2^{*}\mathrm{triv}_i(\alpha') \ar@{=>}"1";[lllu]|-{f(\Xi_{\alpha})^{-1}} \ar@{=>}[ur];"2"|-{f(\Xi_{\alpha'})}}
\end{equation*} 
where we have introduced an object  $c_{\alpha}  \df  \mathrm{triv}_i(s(p))$ where $p=\pi(\pi_1(\alpha))=\pi(\pi_2(\alpha))$ and a 2-morphism $\Xi_{\alpha}  \df (\pi_1(\alpha),s(p),\pi_2(\alpha))$.  It is useful to notice that this means that $f$ is a modification $f \maps  g' \Rightarrow g$. The modification $\psi'$ has the component 
\begin{equation*}
\alxydim{@C=0.5cm}{\mathrm{triv}_i(a) \ar@/^2pc/[rr]^{\id_{\mathrm{triv}_i(a)}}="1" \ar[dr] \ar[rr]|{}="2" && \mathrm{triv}_i(a) \\& c_{\Delta(a)} \ar[ur]& \ar@{=>}"1";"2"|-{\psi(a)} \ar@{=>}"2";[l]|-{f(\Psi)}}
\end{equation*}
at a point $a\in Y$. Finally, the modification
$f'$ has the component
\begin{equation*}
\alxydim{@C=0cm}{&& \pi_2^{*}\mathrm{triv}_i(\Xi) \ar@/^1.5pc/[ddrr]^-{\pi_{23}^{*}g'(\Xi)} \ar[dr] && \\ & c_{\Xi} \ar[rr]^{}="1" \ar@/_1.8pc/[rr]_{\id_{c_{\Xi}}}="2" \ar@{=>}"1";"2"|-{\psi(p)^{-1}} \ar[ur] && c_{\Xi} \ar[dr] & \\ \pi_1^{*}\mathrm{triv}_i(\Xi) \ar@/_1.5pc/[rrrr]_{\pi_{13}^{*}g'(\Xi)} \ar@/^1.5pc/[uurr]^-{\pi_{12}^{*}g'(\Xi)} \ar[rr] \ar[ur] &&c_{\Xi} \ar[rr] && \pi_3^{*}\mathrm{triv}_i(\Xi) \ar@{=>}|-{f(\Psi)}[lluu];"1"}
\end{equation*}
at a point $\Xi\in Y^{[3]}$, where we have introduced the 2-morphism $\Psi \df (c_{\Xi},\pi_2^{*}\mathrm{triv}_i(\Xi),c_{\Xi})$, and $p$ is again the projection of $\Xi$ to $M$.
Now it is straightforward to construct a descent 1-morphism
\begin{equation*}
\rho_{(\mathrm{triv},g,\psi,f)}  \maps (\mathrm{triv},g',\psi',f') \to (\mathrm{triv},g,\psi,f)
\end{equation*}
which consists of the identity \pt\ $h \df \id_{\mathrm{triv}}$ and of a modification $\varepsilon \maps  \pi_2^{*}h \circ g' \Rightarrow g \circ \pi_1^{*}h$ induced from the modification $f \maps  g' \Rightarrow g$ and the left and right \unifier s.  This descent 1-morphism is the component of the \pe\ $\rho \maps \ex{\pi} \circ \mathrm{Rec}_{\pi} \to \id$ we have to construct, at the object $(\mathrm{triv},g,\psi,f)$.

Let us now define the component of $\rho$ at a  descent 1-morphism
\begin{equation*}
(h,\varepsilon)  \maps  (\mathrm{triv}_1,g_1,\psi_1,f_1) \to (\mathrm{triv}_2,g_2,\psi_2,f_2)\text{.}
\end{equation*}
It is useful to introduce a modification $\tilde \epsilon \maps  \bar g_2 \circ \pi_2^{*}h \circ g_1 \Rightarrow \pi_1^{*}h$ where $\bar g_2$ is the pullback of $g_2$ along the map $Y^{[2]} \to Y^{[2]}$ that exchanges the components. It is defined as the following composition of modifications:
\begin{equation*}
\alxydim{@C=2cm}{\bar g_2 \circ \pi_2^{*}h \circ g_1 \ar@{=>}[r]^{\id \circ \varepsilon} & \bar g_2 \circ g_2 \circ \pi_1^{*}h \ar@{=>}[d]^{\Delta_{121}^{*}f_2 \circ \id} \\ & \pi_1^{*}\Delta^{*}g\circ \pi_1^{*}h \ar@{=>}[r]_-{\pi_1^{*}\psi_2^{-1} \circ \id} & \pi_1^{*}\id\circ \pi_1^{*}h \ar@{=>}[r]_-{l_{\pi_1^{*}h}} & \pi_1^{*}h}
\end{equation*}
Now, if we reconstruct and extract local data $(h',\varepsilon')$, the \pt\ $h'$ has the components
\begin{equation*}
h'\quad \maps \quad\alxydim{}{a \ar[r]^{\gamma} &b}\quad\mapsto\quad
\alxydim{@C=-0.4cm@R=0.7cm}{ & i(\mathrm{triv}_1(a)) \ar[dl] \ar[ddd]^{h(a)}="1" \ar[rr]^{i(\mathrm{triv}_1(\gamma))} &\hspace{2cm}& i(\mathrm{triv}_2(b)) \ar@{=>}[dddll]|{h(\gamma)}\ar[dr] \ar[ddd]_{h(b)}="2" \\ i(\pi_2^{*}\mathrm{triv}_1(\alpha)) \ar[d]_{\pi_2^{*}h(\alpha)}="3" &&&& i(\pi_2^{*}\mathrm{triv}_1(\beta)) \ar[d]^{\pi_2^{*}h(\beta)}="4" \\i(\pi_2^{*}\mathrm{triv}_2(\alpha))^2 \ar[dr] &&&& i(\pi_2^{*}\mathrm{triv}_2(\beta)) \ar[dl]\\ & i(\mathrm{triv}_2(a)) \ar[rr]_{i(\mathrm{triv}_2(\gamma))} && i(\mathrm{triv}_2(b)) \ar@{=>}"1";"3"|-{\tilde \varepsilon(\alpha)^{-1}} \ar@{=>}"4";"2"|-{\tilde \varepsilon(\beta)}}
\end{equation*}
with $\alpha  \df  (a,s(\pi(a)))$ and $\beta=(b,s(\pi(b)))$. Like above we observe that $\tilde\varepsilon$ is hence a modification $\tilde\varepsilon  \maps h' \Rightarrow h$. Now, the component $\rho_{(h,\varepsilon)}$ we have to define is a descent 2-morphism
\begin{equation*}
\alxydim{@R=1.3cm@C=1.3cm}{(\mathrm{triv}'_1,g',\psi',f') \ar[d]_{\rho_{(\mathrm{triv}_1,g_1,\psi_1,f_1)}} \ar[r]^-{(h',\varepsilon')}  & (\mathrm{triv}_2',g_2',\psi_2',f_2') \ar@{=>}[dl]|*+{\rho_{(h,\varepsilon)}} \ar[d]^{\rho_{(\mathrm{triv}_2,g_2,\psi_2,f_2)}} \\ (\mathrm{triv}_1,g_1,\psi_1,f_1) \ar[r]_-{(h,\varepsilon)} & (\mathrm{triv}_2,g_2,\psi_2,f_2)\text{,}}
\end{equation*}
this is just a modification $\id \circ h' \Rightarrow h \circ \id$ since the vertical arrows are the identity \pt s. We define  $\rho_{(h,\varepsilon)}$ from $\tilde\varepsilon$ and right and left unifiers in the obvious way.
It is straightforward to see that this defines indeed a descent 2-morphism. Finally, we observe that the definitions $\rho_{(\mathrm{triv},g,\psi,f)}$ and $\rho_{(h,\varepsilon)}$ furnish a \pe\ as required.
\end{proof}

The second part of the proof of Proposition \ref{prop:equiv} is the following lemma.
\begin{lemma}
\label{sec2_lem2}
There is a \pe\ $\id_{\loctrivfunct{i}{\mathrm{Gr}}{2}} \cong \mathrm{Rec}_{\pi} \circ \ex{\pi}$.
\end{lemma}

\begin{proof}
For a 2-functor $F \maps \mathcal{P}_2(X) \to T$ and a $\pi$-local $i$-trivialization $(\mathrm{triv},t)$, let $(\mathrm{triv},g,\psi,f)$ be the associated descent data. We find a \pt\ 
\begin{equation*}
\eta_F \maps  F \to s^{*}R_{(\mathrm{triv},g,\psi,f)}
\end{equation*}
in the following way. Its component at a point $x\in X$ is the 1-morphism $t(s(x)) \maps  F(x) \to \mathrm{triv}_i(s(x))$ in $T$. To define its component at a path $\gamma \maps x \to y$ we recall that $s(\gamma)$ is a composition of paths $\gamma_{i} \maps  a_i \to b_i$  and jumps $\alpha_i$, so that we can compose $\eta_F(\gamma)$ from the pieces
\begin{equation*}
\alxydim{@C=1.5cm}{\pi^{*}F(a_i) \ar[d]_{\eta_F(a_i)} \ar[r]^{\pi^{*}F(\gamma_{i})
} & \pi^{*}F(b_i) \ar@{=>}[dl]|{t(\gamma_i)} \ar[d]^{\eta_F(b_i)} \\ \mathrm{triv}_i(a_i) \ar[r]_{\mathrm{triv}_i(\gamma)} & \mathrm{triv}_i(b_i)}
\quad\text{ and }\quad
\alxydim{}{& F(p) \ar[dl]_{\eta_F({\pi_1(\alpha)})} \ar[d]|{\id}="1" \ar@{=>}"1";[ld]|-{i_t^{-1}} \ar[dr]^{\eta_F({\pi_2(\alpha)})} & \\ \pi_1^{*}\mathrm{triv}_i(\alpha) \ar@/_2pc/[rr]_{g(\alpha)} \ar[r]_-{\pi_1^{*}\bar t(\alpha)} & F(p) \ar[r]_-{\pi_2^{*}t(\alpha)} & \pi_2^{*}\mathrm{triv}_i(\alpha)}
\end{equation*}
where $i_t \maps \bar t \circ t \Rightarrow \id$ is the modification chosen to extract descent data. This defines the \pt\ $\eta_F$ associated to a 2-functor $F$.

Now let $A \maps  F_1 \to F_2$ be a \pt\  between two 2-functors with local trivializations $(\mathrm{triv_1},t_1)$ and $(\mathrm{triv}_2,t_2)$. Let $(h,\varepsilon)$ the associated descent 1-morphism. It is now straightforward to see that
\begin{equation*}
\eta_{A}  \df  i_{t_{1}}^{-1} \maps  \eta_{F_2} \circ A \Rightarrow s^{*}R_{(h,\varepsilon)} \circ \eta_{F_1}
\end{equation*} 
defines a modification in such a way that both definitions together yield  a \pt\ $\eta \maps  \id_{\loctrivfunct{i}{\mathrm{Gr}}{2}} \to \mathrm{Rec}_{\pi} \circ \ex{\pi}$. It is clear that $\eta$ is even a \pe.
\end{proof}

\subsection{Equivalence in the Direct Limit}

As mentioned in Section \ref{sec:loctriv}, path 2-groupoids come with 2-functors $f_{*} \maps  \mathcal{P}_2(M) \to \mathcal{P}_2(N)$ associated to smooth maps $f \maps M \to N$. In turn, these define 2-functors
\begin{equation*}
f^{*} \maps  \mathrm{Funct}(\mathcal{P}_2(N),T) \to \mathrm{Funct}(\mathcal{P}_2(M),T)\text{.}
\end{equation*}
The compatibility \erf{eq:pullbackcomp} of the 2-functors $f_{*}$ with the composition of smooth maps show that
\begin{equation}
\label{eq:pullbackfunct}
(g \circ f)^{*} = f^{*} \circ g^{*}
\end{equation}
for $g \maps  N \to O$ another smooth map.

Now let $\pi_1 \maps Y_1 \to M$ and $\pi_2 \maps  Y_2 \to M$ be surjective submersions and let $\xi  \maps  Y_1 \to Y_2$ be a smooth map such that $\pi_2 \circ \xi = \pi_1$. We call $\xi$ a \emph{refinement} of $\pi_2$. Equation \erf{eq:pullbackfunct} implies that we obtain induced \quot{restriction} 2-functors
\begin{equation}
\label{eq:res2funct}
\mathrm{res}_{\xi} \maps  \mathrm{Triv}^2_{\pi_2}(i)  \to \mathrm{Triv}^2_{\pi_1}(i)
\quand
\mathrm{res}_{\xi} \maps  \mathfrak{Des}^2_{\pi_2}(i) \to \mathfrak{Des}^2_{\pi_1}(i)\text{,}
\end{equation}   
and that these 2-functors themselves satisfy the compatibility condition \erf{eq:pullbackfunct}, with respect to iterated refinements of surjective submersions.

In general, suppose that $S$ is a family of 2-categories parameterized by surjective submersions over a smooth manifold $M$. That is, if $\pi \maps Y \to M$ is a surjective submersion, then $S(\pi)$ is a 2-category. Suppose further that $F$ is a family of \quot{refinement} 2-functors parameterized by refinements of surjective submersions. That is, if $\zeta \maps Y' \to Y$  is a refinement of a surjective submersion $\pi \maps Y \to M$ by $\pi' \maps Y' \to M$, then  $F(\zeta) \maps  S(\pi) \to S(\pi')$ is a 2-functor. Further, we require that $F(\zeta' \circ \zeta) = F(\zeta') \circ F(\zeta)$ for iterated refinements. In this situation, one can form the \emph{direct limit 2-category}
\begin{equation*}
S_M  \df  \lim_{\overrightarrow{\pi}}  S(\pi)\text{.}
\end{equation*}

We shall briefly describe a concrete model for this 2-category, the so-called Grothendieck construction. The precise form can be deduced from the general colimit description in $\infty$-categories; see \cite[Corollary 3.3.4.6]{lurie2009}. An object of $S_M$ is a pair $(\pi,X)$ consisting of a surjective submersion $\pi \maps Y \to M$ and an object $X$ in $S(\pi)$. A \emph{common refinement} of surjective submersions $\pi_1$, $\pi_2$ is a commutative diagram
\begin{equation*}
\alxydim{}{&Z \ar[dd]|{\zeta} \ar[dl]_{y_1} \ar[dr]^{y_2} & \\ Y_1 \ar[dr]_{\pi_1} && Y_2 \ar[dl]^{\pi_2} \\ & M&}
\end{equation*}
in which all maps are surjective submersions. A 1-morphism between objects $(\pi_1,X_1)$ and $(\pi_2,X_2)$ is a common refinement $\zeta$ together with a 1-morphism
\begin{equation*}
f \maps F(y_1)(X_1) \to F(y_2)(X_2)
\end{equation*}
in $S(\zeta)$.
The composition of two 1-morphisms
\begin{equation*}
(\zeta_{12},f_{12}) \maps (\pi_1,X_1) \to (\pi_2,X_2)
\quand
(\zeta_{23},f_{23}) \maps (\pi_2,X_2) \to (\pi_3,X_3)
\end{equation*}
is defined as follows. We consider the fibre product $Z_{13}  \df  Z_{12} \times_{Y_2} Z_{23}$ as a common refinement $\zeta_{13} \maps Z_{13} \to M$ of $\pi_1$ and $\pi_3$. Then, we set
\begin{equation*}
(\zeta_{23},f_{23}) \circ (\zeta_{12},f_{12})  \df  (\zeta_{13}, F(\mathrm{pr}_{Z_{12}})(f_{23}) \circ F(\mathrm{pr}_{Z_{23}})(\zeta_{12}))\text{.}
\end{equation*}
In order to define 2-morphisms between 1-morphisms $(\zeta,f)$ and $(\zeta',f')$ we consider pairs $(\omega,\alpha)$ of a common refinement $\omega \maps W \to M$ of $\zeta$ and $\zeta'$ together with a 2-morphism
\begin{equation*}
\alpha  \maps  F(z)(f) \Rightarrow F(z')(f')
\end{equation*}
in $S(\omega)$, where $z \maps W \to Z$ and $z' \maps W \to Z'$ are the two refinement maps. A 2-morphism is then an equivalence class of pairs $(\omega,\alpha)$, where two pairs $(\omega_1,\alpha_1)$ and $(\omega_2,\alpha_2)$  are identified if the 2-morphisms agree when pulled back to the fibre product $W_1 \times_{Z \times_M Z'} W_2$.

In the present situation, we form the direct limits
\begin{equation*}
\mathrm{Triv}^2(i)_{M}  \df  \lim_{\overrightarrow{\pi}} \mathrm{Triv}^2_{\pi}(i)
\quand
\mathfrak{Des}^2(i)_{M}  \df  \lim_{\overrightarrow{\pi}} \mathfrak{Des}^2_{\pi}(i)\text{.}
\end{equation*}
One checks by inspection that the 2-functor $\ex{\pi}$ commutes with the restriction 2-functors $\mathrm{rec}_{\xi}$ of \erf{eq:res2funct}, so that    the  it induces a 2-functor\begin{equation}
\label{eq:equivlimit}
\mathrm{Ex} \maps  \mathrm{Triv}^2(i)_{M} \to \mathfrak{Des}^2(i)_{M}\text{.}
\end{equation}
Since limits preserve equivalences, we conclude from Proposition \ref{prop:equiv}:

\begin{proposition}
\label{prop:equivlimit}
The 2-functor \erf{eq:equivlimit}
is an equivalence of 2-categories.
\end{proposition}

Now we look at the full sub-2-category $\mathrm{Funct}_i(\mathcal{P}_2(M),T)$ of $\mathrm{Funct}(\mathcal{P}_2(M),T)$ over those  2-functors $F \maps  \mathcal{P}_2(M) \to T$ that \emph{admit} a $\pi$-local $i$-trivialization, for some surjective submersion $\pi \maps Y \to M$. We have a 2-functor
\begin{equation*}
v \maps \mathrm{Triv}^2(i)_M \to \mathrm{Funct}_i(\mathcal{P}_2(M),T) 
\end{equation*}
which simply forgets the local trivialization which is attached to the objects on the left hand side. The 2-functor $v$ is obviously an equivalence of 2-categories, since it is essentially surjective and the identity on Hom-categories. Summarizing, we get:

\begin{theorem}
\label{th:main}
There is an equivalence
\begin{equation*}
\mathfrak{Des}^2(i)_{M} \cong\mathrm{Funct}_i(\mathcal{P}_2(M),T) 
\end{equation*}
between the 2-category of descent data and the 2-category of locally $i$-trivializable 2-functors, realized by a span of equivalences of 2-categories. 
\end{theorem}

Theorem \ref{th:main} is the main result of this article. In \cite{schreiber2} we restrict it to an equivalence between important sub-2-categories: the one of \emph{smooth descent data} (on the left hand side), and the one of \emph{transport 2-functors} (on the right hand side). Transport 2-functors are an axiomatic formulation of connections on non-abelian gerbes.

\setsecnumdepth{1}

\begin{appendix}

\section{Basic 2-Category Theory}
\label{app1}

We introduce notions and facts that we need in this article. For a more complete introduction to 2-categories, see e.g.  \cite{leinster1}.

\begin{definition}
\label{app2cat_def2}
A (small) \emph{2-category} consists of a set of objects, for each pair $(X,Y)$  of objects a set of 1-morphisms denoted $f \maps X \to Y$ and for each pair $(f,g)$ of 1-morphisms $f,g \maps X \to Y$ a set  of 2-morphisms denoted $\varphi \maps f \Rightarrow g$, together with the following structure:
\begin{enumerate}

\item
For every pair $(f,g)$ of 1-morphisms $f \maps  X \to Y$ and $g \maps Y \to Z$, a 
1-morphism $g \circ f \maps X \to Y$, called the composition of $f$ and $g$.

\item 
For every  triple $(f,g,h)$ of 1-morphisms $f \maps W \to X$, $g \maps  X \to Y$ and $h \maps  Y \to Z$, a 2-morphism
\begin{equation*}
a_{f,g,h}  \maps  (h \circ g) \circ f \Rightarrow h \circ (g \circ f)
\end{equation*}
called the associator of $f$, $g$ and $h$.

\item
For every object $X$, a 1-morphism $\id_X \maps X \to X$, called the identity 1-morphism of $X$.

\item
For every 1-morphism $f \maps X \to Y$, 2-morphisms $l_f \maps  f \circ \id_X \Rightarrow f$ and $r_f \maps  \id_Y \circ f \Rightarrow f$, called the left and the right \unifier.
\item
For every pair $(\varphi,\psi)$ of 2-morphisms $\varphi \maps  f \Rightarrow g$ and $\psi \maps  g \Rightarrow h$, a 2-morphism $\psi \bullet \varphi \maps  f \Rightarrow h$, called the vertical composition of $\varphi$ and $\psi$.

\item
For every 1-morphism $f$, a 2-morphism $\id_{f} \maps f \Rightarrow f$, called the identity 2-morphism  of $f$.

\item
For every triple $(X,Y,Z)$ of objects, 1-morphisms $f,f' \maps X \to Y$ and  $g,g' \maps Y \to Z$, and every pair $(\varphi,\psi)$ of 2-morphisms $\varphi \maps f \Rightarrow f'$ and $\psi \maps g \Rightarrow g'$, a 2-morphism $\psi \circ \varphi \maps  g \circ f \Rightarrow g' \circ f'$, called the horizontal composition of $\varphi$ and $\psi$.

\end{enumerate}
This structure has to satisfy the following list of axioms:
\begin{enumerate}
\item[(C1)] The vertical composition of 2-morphisms is associative,
\begin{equation*}
(\phi \bullet \varphi) \bullet \psi =\phi \bullet (\varphi \bullet \psi)
\end{equation*} 
whenever these compositions are well-defined, while the horizontal composition is compatible with the associator in the sense that the diagram
\begin{equation*}
\alxydim{@C=1.8cm@R=1.8cm}{(h \circ g) \circ f \ar@{=>}[d]_{a_{f,g,h}} \ar@{=>}[r]^-{(\psi \circ \varphi) \circ \phi} & (h' \circ g') \circ f' \ar@{=>}[d]^{a_{f',g',h'}} \\ h \circ (g \circ f) \ar@{=>}[r]_-{\psi \circ (\varphi \circ \phi)} & h' \circ ( g' \circ f')}
\end{equation*}
is commutative. 

\item[(C2)] 
The identity 2-morphisms are units with respect to vertical composition,
\begin{equation*}
\varphi \bullet \id_f = \id_g \bullet \varphi
\end{equation*}
for every 2-morphism $\varphi \maps f \Rightarrow g$, while the identity 1-morphisms are compatible with the \unifier s and the associator in the sense that the diagram
\begin{equation*}
\alxydim{@R=1.1cm@C=0.4cm}{(g \circ \id_Y) \circ f \ar@{=>}[rr]^{a_{f,\id_Y,g}} \ar@{=>}[dr]_{l_g \circ \id_f} && g \circ (\id_Y \circ f) \ar@{=>}[dl]^{\id_g \circ r_f} \\ & g \circ f}
\end{equation*}
is commutative. Horizontal composition preserves the identity 2-morphisms in the sense that
\begin{equation*}
\id_g \circ \id_f = \id_{g \circ f}\text{.}
\end{equation*}

\item[(C3)]
Horizontal and vertical  compositions are compatible in the
sense that 
\begin{equation*}
(\psi_1 \bullet \psi_2) \circ (\varphi_1 \bullet \varphi_2) = (\psi_1 \circ\varphi_1)\bullet(\psi_2\circ\varphi_2)
\end{equation*}
whenever these compositions are well-defined. 

\item[(C4)]
All associators and \unifier s are invertible 2-morphisms and  natural in $f$, $g$ and $h$, and the associator satisfies the pentagon axiom
\begin{equation*}
\alxydim{@C=-0.7cm@R=1.3cm}{&&& ((k \circ h) \circ g) \circ f \ar@{=>}[dlll]_{a_{g,h,k} \circ \id_f} \ar@{=>}[drrr]^{a_{f,g,k\circ h}} &&& \\ (k \circ (h \circ g)) \circ f \ar@{=>}[drr]_{a_{f,h \circ g,k}} &&&&&&(k \circ h) \circ (g \circ f) \ar@{=>}[dll]^{a_{g\circ f,h,k}} \\ && k \circ (( h \circ g ) \circ f) \ar@{=>}[rr]_{\id_k \circ a_{f,g,h}} && k \circ (h \circ (g \circ f))\text{.} &&}
\end{equation*}
\end{enumerate}
\end{definition}

In (C4) we have called a 2-morphism $\varphi \maps f \Rightarrow g$ \emph{invertible} or 2-\emph{iso}morphism, if there exists a 2-morphism $\psi \maps  g \Rightarrow f$ such that $\psi \bullet \varphi = \id_{f}$ and $\varphi\bullet\psi=\id_{g}$. The axioms imply a \emph{coherence theorem}: all diagrams of 2-morphisms whose arrows are labelled by associators,  right or left \unifier s, and identity 2-morphisms, are commutative.
A 2-category is  called \emph{strict}, if
\begin{equation*}
(h \circ g) \circ f = h \circ (g \circ f)
\quad\text{ and }\quad
a_{f,g,h}=\id_{h \circ g \circ f}
\end{equation*}
for all triples $(f,g,h)$ of composable 1-morphisms, and if
\begin{equation*}
f \circ \id_X = f = \id_Y \circ f
\quad\text{ and }\quad
r_f = l_f =  \id_f
\end{equation*}
for all 1-morphisms $f$. 
Strict 2-categories allow us to draw pasting diagrams, since multiple compositions of 1-morphisms are well-defined without putting brackets. Pasting diagrams are often more instructive than commutative diagrams of 2-morphisms. 
For an explicit discussion of the strict case the reader is referred to Appendix A.1 in \cite{schreiber5}.

\begin{example}
\label{app2cat_ex2}
Let $\mathfrak{C}$ be a monoidal category, i.e. a category equipped with a functor $\otimes \maps  \mathfrak{C} \times \mathfrak{C} \to \mathfrak{C}$, a distinguished object $\trivlin$  in $\mathfrak{C}$, a natural transformation $\alpha$ with components
\begin{equation*}
\alpha_{X,Y,Z} \maps  (X \otimes Y) \otimes Z \to X \otimes (Y \otimes Z)\text{,}
\end{equation*}
and natural transformations $\rho$ and $\lambda$ with components
\begin{equation*}
\rho_X \maps  \trivlin \otimes X \to X
\quad\text{ and }\quad
\lambda_X \maps  X \otimes \trivlin \to X
\end{equation*}
which are subject to the usual coherence conditions, see, e.g. \cite{maclane2}. The monoidal category $\mathfrak{C}$ defines a 2-category $\mathcal{B}\mathfrak{C}$ in the following way: it has a single object,  the 1-morphisms are the objects of $\mathfrak{C}$ and the 2-morphisms between two 1-morphisms $X$ and $Y$ are the morphisms $f \maps X \to Y$ in $\mathfrak{C}$. The composition of 1-morphisms and the horizontal composition is the tensor product $\otimes$, and the associator $a_{X,Y,Z}$ is given by the component $\alpha_{Z,Y,X}$. The identity 1-morphism is the tensor unit $\trivlin$, and the \unifier s are given by the natural transformations $\rho$ and $\lambda$. The vertical composition  and the identity are just the ones of $\mathfrak{C}$. It is straightforward to check that axioms (C1) to (C4) are either satisfies due to the axioms of the category $\mathfrak{C}$, the functor $\otimes$, or the natural transformations $\alpha$, $\rho$ and $\lambda$, or due to the coherence axioms.
The 2-category $\mathcal{B}\mathfrak{C}$ is strict if and only if the monoidal category $\mathfrak{C}$ is strict. 
\end{example}
 
 In any 2-category, a 1-morphism
$f \maps  X \to Y$ is called \emph{invertible}, if there exists
another 1-morphism $g \maps  Y \to X$ together with natural 2-isomorphisms $i
 \maps  g \circ f \Rightarrow \id_X$ and $j  \maps  \id_Y \Rightarrow f \circ
g$ such that the diagrams
\begin{equation}
\label{app2cat_1}
\alxydim{@C=1.5cm@R=1cm}{(f \circ g) \circ f \ar@{=>}[d]_{a_{f,g,f}} \ar@{=>}[r]^-{j^{-1} \circ \id_f} \ar@{=>}[d]_{} & \id_Y \circ f \ar@{=>}[dd]^{r_f} \\ f \circ (g \circ f) \ar@{=>}[d]_-{\id_f \circ i} & \\ f \circ \id_X \ar@{=>}[r]_-{l_f} & f}
\quad\text{ and }\quad
\alxydim{@C=1.5cm@R=1cm}{(g \circ f) \circ g \ar@{=>}[d]_{a_{g,f,g}} \ar@{=>}[r]^-{i \circ \id_g} \ar@{=>}[d]_{} & \id_X \circ g \ar@{=>}[dd]^{r_g} \\ g \circ (f \circ g) \ar@{=>}[d]_-{\id_g \circ j^{-1}} & \\ g \circ \id_Y \ar@{=>}[r]_-{l_g} & g}
\end{equation}
are commutative. 
Let us remark that neither in the strict nor in the general case the inverse 1-morphism $g$ is uniquely determined. We call a triple $(g,i,j)$ a \emph{weak inverse} of $f$. By \emph{1-isomorphism} we mean an invertible 1-morphism together with a weak inverse.  

\begin{remark}
Often a 2-category is called bicategory, while a strict 2-category is called 2-category. Invertible 1-morphisms are often  called adjoint equivalences.
\end{remark}

\begin{definition}
\label{app2cat_def1}
A (strict) 2-category in which every 1-morphism and every 2-morphism is invertible, is called  (strict) \emph{2-groupoid}.
 \end{definition}

The following definition generalizes the one of  a functor between categories.

\begin{definition}
\label{app2cat_def4}
Let $S$ and $T$ be two 2-categories.
A \emph{2-functor} $F \maps S \to T$ assigns
\begin{enumerate}
\item 
an object $F(X)$ in $T$ to each object $X$ in $S$,
\item a 1-morphism $F(f) \maps F(X) \to F(Y)$ in $T$ to each 1-morphism $f \maps X \to Y$ in $S$, and 
\item
a 2-morphism $F(\varphi) \maps F(f) \Rightarrow F(g)$ in $T$ to each 2-morphism $\varphi \maps  f \Rightarrow g$ in $S$. 
\end{enumerate}
Furthermore, it has
\begin{enumerate}
\item[(a)]
a  2-isomorphism $u_{X} \maps F(\id_X) \Rightarrow \id_{F(X)}$
in $T$ for each object $X$ in $S$, and
\item[(b)]
a  2-isomorphism
$c_{f,g} \maps  F(g) \circ F(f) \Rightarrow F(g\circ f)$
in $T$  for each  pair of composable 1-morphisms $f$ and $g$ in S. 
\end{enumerate}
Four axioms
have to be satisfied:
\begin{enumerate}
\item[(F1)]
The vertical composition is respected in the sense that
\begin{equation*}
F(\psi \bullet \varphi) = F(\psi) \bullet F(\varphi)
\quad\text{ and }\quad
F(\id_f) = \id_{F(f)}
\end{equation*}
for all composable 2-morphisms $\varphi$ and $\psi$, and any 1-morphism $f$.
\item[(F2)]
The horizontal composition is respected in the sense that the diagram
\begin{equation*}
\alxydim{@C=1.8cm@R=1.3cm}{F(g) \circ F(f) \ar@{=>}[r]^{F(\psi) \circ F(\varphi)} \ar@{=>}[d]_{c_{f,g}} & F(g') \circ F(f') \ar@{=>}[d]^{c_{f',g'}} \\ F(g \circ f) \ar@{=>}[r]_{F(\psi \circ \varphi)} & F(g' \circ f')}
\end{equation*}
is commutative for all horizontally composable 2-morphisms $\varphi$ and $\psi$. 

\item[(F3)] 
The compositor $c_{f,g}$ is compatible with the associators of $S$ and $T$ in the sense that the diagram
\begin{equation*}
\alxydim{@C=2.4cm}{(F(h) \circ F(g)) \circ F(f) \ar@{=>}[r]^{a_{F(f),F(g),F(h)}} \ar@{=>}[d]_{c_{g,h}\circ \id_{F(f)}} & F(h) \circ (F(g) \circ F(f)) \ar@{=>}[d]^{\id_{F(h)} \circ c_{f,g}} \\ F(h \circ g) \circ F(f) \ar@{=>}[d]_{c_{f,h \circ g}} & F(h) \circ F(g \circ f) \ar@{=>}[d]^{c_{g \circ f,h}} \\ F((h \circ g) \circ f) \ar@{=>}[r]_{F(a_{f,g,h})} & F(h \circ (g \circ f))}
\end{equation*}
is commutative for all composable 1-morphisms $f$, $g$ and $h$.

\item[(F4)]
Compositor and unitor are compatible with the \unifier s of $S$ and $T$ in the sense that the diagrams
\begin{equation*}
\alxydim{@C=0.9cm@R=1.3cm}{F(f) \circ F(\id_X) \ar@{=>}[r]^-{c_{\id_X,f}} \ar@{=>}[d]|{\id_{F(f)} \circ u_X} & F(f \circ \id_X) \ar@{=>}[d]|{F(l_{f})} \\ F(f) \circ \id_{F(X)} \ar@{=>}[r]_-{l_{F(f)}}  & F(f)}
\text{ and }
\alxydim{@C=0.9cm@R=1.3cm}{F(\id_Y) \circ F(f) \ar@{=>}[r]^-{c_{f,\id_Y}} \ar@{=>}[d]|{u_Y \circ \id_{F(f)}} & F(\id_Y \circ f) \ar@{=>}[d]|{F(r_{f})} \\ \id_{F(Y)}\circ F(f) \ar@{=>}[r]_-{r_{F(f)}}  & F(f)}
\end{equation*}
are commutative for every 1-morphism $f$.
\end{enumerate}
\end{definition}
Sometimes we represent a 2-functor $F \maps S \to T$ diagrammatically as an assignment
\begin{equation*}
F \quad:\quad 
\bigon{X}{Y}{f}{g}{\varphi}
\quad\longmapsto\quad
\bigon{F(X)}{F(Y)}{F(f)}{F(g)}{F(\varphi)}
\text{.}
\end{equation*}
In case that the 2-category $T$ is strict, and the axioms (F2) to (F4)  can be expressed by pasting diagrams in the following way:
\begin{itemize}
\item 
Axioms (F2) is equivalent to the equality
\begin{equation*}
\alxydim{}{&F(Y)  \ar@/_1.5pc/[rd]|-{F(g')}="4" \ar@/^1.3pc/[dr]^{F(g)}="3"\\F(X) \ar@/_3pc/[rr]_{F(g' \circ f')}="5" \ar@{=>}[ur];"5"|>>>>>>>{c_{f',g'}} \ar@/_1.5pc/[ru]|-{F(f')}="2" \ar@/^1.3pc/[ru]^{F(f)}="1" &  \ar@{=>}"1";"2"|{F(\varphi)}\ar@{=>}"3";"4"|{F(\psi)} & F(Z)}
=
\alxydim{}{&F(Y)\ar@/^1.2pc/[dr]^{F(g)}="7"& \\ F(X) \ar@/_3pc/[rr]_{F(g' \circ f')}="9"  \ar[rr]|{F(g \circ f)}="8" \ar@{=>}[ur];"8"|{c_{f,g}} \ar@/^1.2pc/[ru]^{F(f)}="6" &  & F(Z)\text{.} \ar@{=>}"8";"9"|{F(\psi \circ \varphi)}}
\end{equation*}
\item
Axiom (F3) is equivalent to the \tetraeder identity
\begin{equation*}
\alxydim{@C=1.5cm@R=1.7cm}{F(X) \ar[r]^{F(g)} \ar@{<-}[d]_{F(f)}
& F(Y) \ar[d]^{F(h)} \\ F(W)  \ar[ur]|{F(g \circ f)}="2" \ar@{=>}[u];"2"|{c_{f,g}} \ar[r]_{F(h \circ g \circ f)}="1"
&  \ar@{<=}"1";"2"|{c_{g \circ f,h}} F(Z)}
=
\alxydim{@C=1.5cm@R=1.7cm}{F(X) \ar[dr]|{F(h \circ g)}="2" \ar[r]^{F(g)} \ar@{<-}[d]_{F(f)}
& F(Y) \ar[d]^{F(h)}  \ar@{<=}"2";[]|{c_{h,g}}
 \\  F(W)  \ar[r]_{F(h \circ g \circ f)}="1"
\ar@{<=}"1";"2"|{c_{{f,h \circ g}}}& F(Z)\text{.}}
\end{equation*}
\item
Axiom (F4) is equivalent to the equalities
\begin{equation*}
c_{\id_X,f} = \id_{F(f)} \circ u_X
\quad\text{ and }\quad
c_{f,\id_Y} = u_Y \circ \id_{F(f)}\text{.}
\end{equation*}
\end{itemize}

We use two layers of strictness conditions for 2-functors, normally in a situation when $S$ and $T$ are strict 2-categories. Firstly, we call a 2-functor $F: S \to T$  \emph{normalized} if 
\begin{equation}
\label{respectunitors}
F(\id_X) = \id_{F(X)}
\quad\text{ and }\quad
u_X=\id_{\id_{F(X)}}
\end{equation}
for all objects $X$ in $S$,
 and if
\begin{equation}
\label{normalized2functor}
F(g) \circ F(f)= \id_{F(X)}
\quad\text{ and }\quad
 c_{f,g}=\id_{\id_X}
\end{equation} 
for all 1-morphisms  $f:X \to Y$ and $g:Y \to X$ such that $g\circ f=\id_X$. Roughly speaking, normalized 2-functors strictly respect identities and inverses. The second and stronger requirement is that the 2-functor $F \maps S \to T$ is  \emph{strict}: this requires \erf{respectunitors} while \erf{normalized2functor}   is superseded by the condition that
\begin{equation*}
F(g) \circ F(f)= F(g\circ f)
\quad\text{ and }\quad
 c_{f,g}=\id_{F(g \circ f)}
\end{equation*} 
for \emph{all} composable 1-morphisms $f$ and $g$.
In  case of strict 2-functors  between strict 2-categories only axioms (F1) and (F2) remain, claiming that both compositions are respected.

The following definition 
generalizes a natural transformation between  functors.

\begin{definition}
\label{app1_def1}
Let $F_1$ and $F_2$ be two  2-functors from $S$ to $T$.
A \emph{pseudonatural transformation}
$\rho\maps F_1 \to F_2$ assigns
\begin{enumerate}
\item 
a 1-morphism $\rho(X) \maps  F_1(X) \to F_2(X)$ in $T$ to each object $X$ in $S$, and
\item
a 2-isomorphism $\rho(f) \maps  \rho(Y) \circ F_1(f) \Rightarrow F_2(f) \circ \rho(X)$ in $T$ to each 1-morphism $f \maps X \to Y$ in $S$,
\end{enumerate}
such that two axioms are satisfied:
\begin{enumerate}
\item[(T1)]
The composition of 1-morphisms in $S$ is respected in the sense that the diagram
\begin{equation*}
\alxydim{@C=2.8cm}{(\rho(Z) \circ F_1(g)) \circ F_1(f) \ar@{=>}[r]^-{a_{F_1(f),F_1(g),\rho(Z)}} \ar@{=>}[d]_{\rho(g) \circ \id_{F_1(f)}} & \rho(Z) \circ (F_1(g) \circ F_1(f)) \ar@{=>}[d]^{ \id_{\rho(Z)}\circ (c_1)_{f,g}} \\ (F_2(g) \circ \rho(Y)) \circ F_1(f) \ar@{=>}[d]_{a_{F_1(f),\rho(Y),F_2(g)}} & \rho(Z) \circ F_1(g \circ f) \ar@{=>}[d]^{\rho(g \circ f)} \\ F_2(g) \circ ( \rho(Y) \circ F_1(f)) \ar@{=>}[d]_{\id_{F_2(g)} \circ \rho(f)} & F_2(g \circ f) \circ \rho(X) \ar@{=>}[d]^{(c_2)^{-1}_{f,g} \circ \id_{\rho(X)}} \\ F_2(g) \circ (F_2(f) \circ \rho(X)) \ar@{=>}[r]_{a^{-1}_{\rho(X),F_2(f),F_2(g)}} & (F_2(g) \circ F_2(f)) \circ \rho(X)}
\end{equation*}
is commutative for all composable 1-morphisms $f$ and $g$. Here, $a$ is the associator of the 2-category $T$ and $c_1$ and $c_2$ are the compositors of the 2-functors $F_1$ and $F_2$, respectively.

\item[(T2)]
It is natural  in the sense that the diagram
\begin{equation*}
\alxydim{@C=1.8cm@R=1.6cm}{\rho(Y) \circ F_1(f) \ar@{=>}[r]^-{\rho(f)} \ar@{=>}[d]_{\id_{\rho(Y)}\circ F_1(\varphi)} & F_2(f) \circ \rho(X) \ar@{=>}[d]^{F_2(\varphi) \circ \id_{\rho(X)}} \\ \rho(Y) \circ F_1(g) \ar@{=>}[r]_-{\rho(g)} & F_2(g) \circ \rho(X)}
\end{equation*}
is commutative for all 2-morphisms $\varphi \maps f \Rightarrow g$.

\end{enumerate}
\end{definition}

If one considers a version of \pt s where the 2-morphisms $\rho(f)$ do not have to be invertible, there is a third axiom related to the value of $\rho$ at the identity 1-morphism $\id_X$ of an object $X$ in $S$. In our setup this axiom is automatically satisfied, as the following lemma shows. 

\begin{lemma}
\label{app2cat_lem1}
Let $\rho \maps F_1 \to F_2$ be a \pt\ between 2-functors $F_1$ and $F_2$ with unitors $u^1$ and $u^2$, respectively. Then, the following assertions hold. 
\begin{enumerate}[(i)]
\item 
The diagram
\begin{equation*}
\alxydim{@C=1.4cm@R=1.3cm}{\rho(X) \circ F_1(\id_X) \ar@{=>}[rr]^{\rho(\id_X)} \ar@{=>}[d]_{\id_{\rho(X)} \circ u^1_X} && F_2(\id_X) \circ \rho(X) \ar@{=>}[d]^{u^2_X \circ \id_{\rho(X)}} \\ \rho(X) \circ \id_{F_1(X)} \ar@{=>}[r]_-{l_{\rho(X)}} & \rho(X) \ar@{=>}[r]_-{r_{\rho(X)}^{-1}} & \id_{F_2(X)} \circ  \rho(X)}
\end{equation*}
is commutative. 

\item
If $S$ and $T$ are strict 2-categories, and  $F_1$ and $F_2$ are normalized 2-functors, then 
\begin{equation*}
\rho(\id_X)=\id_{\rho(X) \circ F_1(\id_X)} = \id_{F_2(\id_X) \circ \rho(X)}
\end{equation*}
for all objects $X$ in $S$, and
\begin{equation*}
\rho(g) \circ \id_{F_1(f)} = \id_{F_2(g)} \circ \rho(f)^{-1} 
\end{equation*}
for all 1-morphisms $f,g$ in $S$ with $g \circ f = \id$.

\end{enumerate}
\end{lemma}

\begin{proof}
For (i) one applies axiom (T1) to 1-morphisms $f=g=\id_X$. Then one uses axiom (T2) for $\rho$, axiom (F4) for both 2-functors, axiom (C2) for $T$, and the invertibility of the 2-morphism $\rho(g)$ and of the 1-morphism $F_2(\id_X)$. The first claim of (ii) follows from (i) under the strictness assumptions, and the second claim follows from the first claim and  axiom (T1).
\end{proof}

Sometimes we represent a \pt\ $\rho \maps F_1 \to F_2$ diagrammatically by
\begin{equation*}
\rho \quad:\quad
\alxy{ X \ar[r]^{f} & Y}
\quad\longmapsto\quad
\alxydim{@C=1.2cm@R=1.2cm}{F_1(X) \ar[r]^{F_1(f)} \ar[d]_{\rho(X)} & F_1(Y) \ar[d]^{\rho(Y)}
\ar@{=>}[dl]|{\rho(f)} \\ F_2(X) \ar[r]_{F_2(f)} & F_2(Y)\text{,}}
\end{equation*}
and if the 2-category $T$ is strict, the axioms can be expressed by pasting diagrams in the following way:
\begin{itemize}
\item 
Axiom (T1) is equivalent to
\begin{equation*}
\alxydim{@C=1.2cm@R=1.2cm}{F_1(X) \ar[r]^{F_1(f)} \ar[d]_{\rho(X)} & F_1(Y) \ar[r]^{F_1(g)} \ar[d]|{\rho(Y)}
\ar@{=>}[dl]|{\rho(f)} &F_1(Z) \ar@{=>}[dl]|{\rho(g)} \ar[d]^{\rho(Z)} \\ F_2(X) \ar@/_3pc/[rr]_{F_2(g \circ f)}="1" \ar[r]_{F_2(f)} & F_2(Y) \ar@{=>}"1"|{(c_2)_{f,g}}\ar[r]_{F_2(g)} & F_2(Z)}
=
\alxydim{@C=1.2cm@R=1.2cm}{F_1(X) \ar@/^3pc/[r]^{F_1(g) \circ F_1(f)}="1"
  \ar[r]_{F_1(g \circ f)}="2" \ar@{=>}"1";"2"|{(c_1)_{f,g}} \ar[d]_{\rho(X)} & F_1(Z) \ar[d]^{\rho(Z)}
\ar@{=>}[dl]|{\rho(g \circ f)} \\ F_2(X) \ar[r]_{F_2(g \circ
 f)} & F_2(Z)\text{.}}
\end{equation*}
\item
Axiom (T2) is equivalent to
\begin{equation*}
\alxydim{@C=1.2cm@R=1.2cm}{F_1(X) \ar[r]^{F_1(f)} \ar[d]_{\rho(X)} & F_1(Y) \ar[d]^{\rho(Y)}
\ar@{=>}[dl]|{\rho(f)} \\ F_2(X) \ar@/_3pc/[r]_{F_2(g)}="2"   \ar[r]^{F_2(f)}="1" & F_2(Y) \ar@{=>}"1";"2"|{F_2(\varphi)}}
=
\alxydim{@C=1.2cm@R=1.2cm}{F_1(x) \ar@/^3pc/[r]^{F_1(f)}="1" \ar[r]_{F_1(g)}="2"
\ar@{=>}"1";"2"|{F_1(\varphi)} \ar[d]_{\rho(X)} & F_1(Y) \ar[d]^{\rho(Y)}
\ar@{=>}[dl]|{\rho(g)} \\ F_2(X) \ar[r]_{F_2(g)} & F_2(Y)\text{.}}
\end{equation*}
\end{itemize}

We need one more definition for situations where two 
\pt s are present.

\begin{definition}
Let $F_1,F_2 \maps S \to T$ be two  2-functors and let $\rho_1,\rho_2 \maps  F_1\to F_2$
be \pt s. A \emph{modification}
$\mathcal{A} \maps \rho_1 \Rightarrow \rho_2$ assigns a 2-morphism 
\begin{equation*}
\mathcal{A}(X) \maps  \rho_1(X) \Rightarrow \rho_2(X)
\end{equation*}
in $T$ to any object $X$ in $S$,
such that the diagram
\begin{equation}
\label{app2cat_2}
\alxydim{@C=1.6cm@R=1.6cm}{\rho_1(Y) \circ F_1(f) \ar@{=>}[r]^{\rho_1(f)} \ar@{=>}[d]_{\mathcal{A}(Y) \circ \id_{F_1(f)}} & F_2(f) \circ \rho_1(X) \ar@{=>}[d]^{\id_{F_2(f)} \circ \mathcal{A}(X)} \\ \rho_2(Y) \circ F_1(f) \ar@{=>}[r]_{\rho_2(f)} & F_2(f) \circ \rho_2(X)}
\end{equation}
is commutative for every 1-morphism $f$ in $S$.
\end{definition}
In the case that $T$ is a strict 2-category, the latter diagram is equivalent to a pasting diagram, see Definition A.4 in \cite{schreiber5}.

As one might expect, 2-functors, \pt s, and modifications fit  into the structure of a 2-category that we denote by $\mathrm{Funct}(S,T)$.
It is strict if and only if $T$ is strict.
Note that the definition of  invertibility in a 2-category applies;  we  
call a 2-isomorphism in the 2-category $\mathrm{Funct}(S,T)$  \emph{invertible
modification}, and a 1-isomorphism    \emph{\pe}.  This leads to the following 

\begin{definition}
Let $S$ and $T$ be  2-categories. 
A  2-functor $F \maps S \to T$ is called
an \emph{equivalence of 2-categories}, if there exists a   2-functor $G \maps T \to S$ together with \pe s $\rho_S \maps  G \circ F \to \id_S$ and $\rho_T \maps 
F \circ G \to \id_T$. 
\end{definition}

\end{appendix}

\kobib{../../bibliothek/tex}

\end{document}